\documentclass[12pt]{amsart}
\usepackage{ourcommands}
\usepackage{colonequals}
\newcommand{\mm}{\bar{M}_2}
\newcommand{\mone}{\bar{M}_{1,1}}
\title[The Integral Chow Ring of $\mm$]{\boldmath The Integral Chow Ring of $\mm$}

\begin{document}
\maketitle

\begin{abstract}
In this paper we compute the Chow ring of
the moduli stack $\mm$ of stable curves of genus~$2$
with integral coefficients.
\end{abstract}

\section{Introduction}

A fundamental problem in the theory of algebraic curves is to understand
the geometry of the moduli stack $\bar{M}_g$ of stable curves of genus $g$.
One of the first questions we might ask about the geometry of these spaces is:

\begin{quest}
What is the Chow ring $\chow^*(\bar{M}_g)$ of $\bar{M}_g$?  
\end{quest}

Despite much progress over the past half century,
answers to this question are only known after making
various simplifications --- and even then only in small genus.

One such simplification is to study the Chow ring with rational coefficients
$\chow^*(\bar{M}_g) \otimes \qq$;
this removes the subtle torsion phenomenon that exist due to the presence
of loci of curves with automorphisms.
For example, the Chow ring of $\bar{M}_g$ with rational coefficients
is known for $g = 2$ by work of Mumford \cite{enum},
and for $g = 3$ by work of Faber \cite{faber}.

Another such simplification is to replace $\bar{M}_g$
with a simpler but related space, e.g.\ $M_g$ or $\bar{M}_{0, n}$ (with $n \geq 3$) or $\bar{M}_{1, n}$
(with $n \geq 1)$.
For example, the full Chow ring (i.e.\ with integral coefficients) of $M_g$
is known for $g = 2$ by work of Vistoli \cite{vistoliM2}.

However, to date, the full Chow ring $\chow^*(\bar{M}_g)$
is not known in a single case.
The goal of the present paper is to give the first such example:

\begin{thm} \label{thm:main} Over any base field of characteristic
distinct from $2$ and $3$, the Chow ring of the moduli space of stable curves
of genus $2$ is given by
\[\chow^*(\mm) = \zz[\lambda_1, \lambda_2, \delta_1]/(24 \lambda_1^2 - 48 \lambda_2, 20 \lambda_1 \lambda_2 - 4 \delta_1 \lambda_2, \delta_1^3 + \delta_1^2 \lambda_1, 2\delta_1^2 + 2\delta_1 \lambda_1),\]
where $\lambda_1$ and $\lambda_2$ denote the Chern classes of the Hodge bundle,
and $\delta_1$ denotes the class of the boundary substack with a disconnecting
node.

In this basis, the class of the boundary substack with a self-node
is given by $\delta_0 = 10 \lambda_1 - 2 \delta_1$.
\end{thm}

One might hope to compute the Chow ring of $\mm$ using equivariant
intersection theory \cite{EdGr}, as done for $M_2$ by Vistoli in \cite{vistoliM2}.
Unfortunately, finding a suitable presentation of $\mm$ 
seems extremely difficult.

Instead, we stratify $\mm$ into the boundary
divisor $\Delta_1$ and its complement $\mm \smallsetminus \Delta_1$,
whose Chow rings can be computed using equivariant intersection theory.
This gives us a sequence
\[\chow^*(\Delta_1) \to \chow^*(\mm) \to \chow^*(\mm \smallsetminus \Delta_1) \to 0.\]
One key observation is that a variant of Vistoli's argument in \cite{vistoliM2}
can not only compute $\chow^*(M_2)$, but in fact compute both
the Chow ring and the first higher Chow groups with $\ell$-adic coefficients
(which control the failure of left-exactness in the above sequence)
of $\mm \smallsetminus \Delta_1$.
Unfortunately, the first higher Chow groups of the open stratum $\mm \smallsetminus \Delta_1$
do \emph{not} vanish for $\ell = 2$; they are generated by $2$-torsion classes in degrees $4$ and $5$.

For the degree $4$ generator, we observe that
all the $2$-torsion in $\chow^3(\Delta_1)$ arises, in some sense,
``locally'' from those curves in $\Delta_1$ which admit a bielliptic involution
(exchanging the two components).
Using the universal family of curves with a bielliptic involution as a test family,
we show that the pushforward map $\chow^3(\Delta_1) \to \chow^4(\mm)$ is injective,
and so the image of the degree $4$ generator in $\chow^3(\Delta_1)$ must vanish.

By contrast, the degree $5$ generator can be thought of as arising, in some sense,
``globally'' from the presence of the hyperelliptic involution
on every curve; in particular, it seems extremely difficult to study this class using a test family.
Instead, we get a handle on this class using the action of $\mathbb{G}_m$ on $\mm$
given by taking the ``universal quadratic twist with respect to the hyperelliptic involution''.
This $\mathbb{G}_m$-action kills the degree $5$ generator,
which then corresponds to a ``syzygy'' between two relations of smaller degree
in $\chow^*(\mm / \mathbb{G}_m)$;
using this, we can explicitly compute the 
boundary map on this degree $5$ generator.

\medskip

A brief outline of the remainder of the paper is as follows: We begin in Section~\ref{sec:notation}
by fixing some notation for some group representations
that will appear throughout the paper.
Next, in Section~\ref{sec:pres}, we give presentations of the strata
$\Delta_1$ and $\mm \smallsetminus \Delta_1$,
as quotients of open subsets in affine spaces by linear algebraic groups --- a group we term $G$ defined in the next section for $\Delta_1$,
respectively the group $\gl_2$ for $\mm \smallsetminus \Delta_1$.
In Section~\ref{sec:mult}, we give formulas for pushforward
maps between Chow groups of various projective bundles under
multiplication maps, that will be used throughout the remainder of the paper.
Then in Section~\ref{sec:bg}, we compute $\chow^*(BG)$.
This is enough to show in Section~\ref{genrel}
that $\chow^*(\mm)$ is generated by $\delta_1$, $\lambda_1$, and $\lambda_2$,
as claimed, and to establish all the relations appearing in the statement of Theorem~\ref{thm:main}
in characteristic zero.

Then in Section~\ref{sec:pushpull}, we explicitly compute
the pushforward and pullback maps between Chow groups along $BT \to BG$,
where $T$ is the maximal torus of $G$.
In Sections~\ref{sec:adelta1} and~\ref{sec:au},
we compute the Chow ring of the strata $\Delta_1$, respectively the Chow ring
and first higher Chow groups with $\ell$-adic coefficients
of the strata $\mm \smallsetminus \Delta_1$;
we also derive an explicit formula for the image of the degree $5$
generator mentioned above.
In Section~\ref{sec:tofinite}, we use this to reduce Theorem~\ref{thm:main}
to the nonvanishing of a finite number of classes in $\chow^*(\mm) \otimes \zz/2\zz$.
Finally, in Section~\ref{sec:biell}, we compute the Chow ring of the universal
family of bielliptic curves of genus $2$,
which completes the proof of Theorem~\ref{thm:main}
by serving as a test family to show the nonvanishing of these desired classes.

\subsection{Assumptions on Characteristic} For the remainder of the paper,
we work over a field $k$ of characteristic distinct
from $2$ and $3$.

\subsection{Remark} Upon completion of this manuscript, the author
learned that Angelo Vistoli and Andrea Di Lorenzo are working on
a different approach to this problem, which will hopefully yield an
indepenent proof of Theorem~\ref{thm:main}.

\subsection{Acknowledgements}
First and foremost, the author would like to profusely thank Akhil Mathew
for many extremely fruitful conversations (as well as comments on the manuscript
and assistance locating references).
In these conversations, he suggested several
mathematical insights that were critical to the success of this project; these include
the technique used to calculate $\chow^*(\Delta_1)$
in Section~\ref{sec:adelta1}, and the idea of using
the higher Chow groups of $\mm \smallsetminus \Delta_1$.
This paper would not have existed without his help!

The author would also like to thank Ken Ono, Ravi Vakil, and David Zureick--Brown
for various helpful discussions.
Finally, the author would like to acknowledge the generosity of the
National Science Foundation in supporting this research --- both via
the Research Experience for Undergraduates
program where I began thinking about this problem in 2013,
and via the Mathematical Sciences Postdoctoral Research Fellowship
program that provided recent funding for the completion of this project.

\section{Notation \label{sec:notation}} In this section we fix some notational conventions
that we shall use for the remainder of the paper.

\begin{defi} Write $V$ for the standard representation of $\gl_2$,
and define the representations
\[V_n = \sym^n V^* \tand V_n(m) = V_n \otimes (\det)^{\otimes m}.\]
\end{defi}

\begin{defi} \label{G} We let
\[G = (\mathbb{G}_m \times \mathbb{G}_m) \rtimes \zz/2\zz,\]
where the action
of $\zz/2\zz$ on $\mathbb{G}_m \times \mathbb{G}_m$ permutes the factors.

Given a representation $V$ of $\mathbb{G}_m$, we observe that the direct sum
$V \oplus V$ inherits a natural action of $G$ (where the $\zz/2\zz$-action
permutes the factors); we write $V \bpls V$ for this.
\end{defi}

\begin{defi} Write $L_n$ for the $1$-dimensional representation of $\mathbb{G}_m$
where $z$ acts as multiplication by $z^n$, and let
\[L_{a_1, a_2, \ldots, a_n} = L_{a_1} \oplus L_{a_2} \oplus \cdots \oplus L_{a_n}.\]
Let $W_{a_1, a_2, \ldots, a_n}$ be the representation of $G$ defined by
\[W_{a_1, a_2, \ldots, a_n} = L_{a_1, a_2, \ldots, a_n} \bpls L_{a_1, a_2, \ldots, a_n},\]
and denote $W = W_1$.

Additionally, let $\Gamma$ be the representation of $G$ arising from the
sign representation of the $\zz/2\zz$ quotient of $G$.
\end{defi}

\begin{defi} Write
\begin{align*}
\alpha_i &= c_i(V) \in \chow^*(B \! \gl_2) \\
\beta_i &= c_i(W) \in \chow^*(BG) \\
\gamma &= c_1(\Gamma) \in \chow^*(BG).
\end{align*}
\end{defi}

\section{\label{sec:pres} Presentations of $\Delta_1$ and $\mm \smallsetminus \Delta_1$}

In this section, we give presentations of the boundary stratum $\Delta_1$
of curves with a disconnecting node, and of its complement
$\mm \smallsetminus \Delta_1$, as the quotient of an affine space
by the action of an algebraic group.
These presentations will be used in subsequent sections
to compute the Chow rings of these loci.

\subsection{\label{sec:s} \boldmath A presentation of $\Delta_1$}

Note that the curves parameterized by $\Delta_1$ are in the form
of two elliptic curves glued together at their marked points.
In other words, the stack $\Delta_1$ can be described as the symmetric square
of $\mone$:
\[\Delta_1 \simeq \sym^2 \mone.\]
Recall that
an elliptic curve $E$ can be written in Weierstrass form:
\[y^2 = x^3 + ax + b,\]
which is stable if and only if $a$ and $b$ do not simultaneously vanish.
The fiber of the Hodge bundle here is given by
\[H^0(\omega_E) = \left\langle \frac{dx}{y}\right \rangle.\]
Moreover, the isomorphisms between two curves in
Weierstrass form
are given by
\[(x,y) \mapsto (u^{-2} x, u^{-3} y) \tfor u \in \mathbb{G}_m;\]
the isomorphism given by $u$ defines an isomorphism
between the curves $y^2 = x^3 + a x + b$
and $y^2 = x^3 + u^4 a x + u^6 b$, and acts
on $dx/y$ as multiplication by $u$.

Thus, $\mone \simeq (L_{4,6} \smallsetminus 0) / \mathbb{G}_m$. 
Moreover, the Hodge bundle of $\mone$ is the pullback of the
representation $L_1$ from $B \mathbb{G}_m$.
Consequently, we have
the following fundamental presentation:
\begin{equation} \label{sym2pres} \Delta_1 \simeq \sym^2 \mone \simeq  \left(W_{4,6} \smallsetminus (L_{4,6} \times 0 \cup 0 \times L_{4, 6})\right) / G,
\end{equation}
and the Hodge bundle is the pullback from $BG$ of the representation
$W_1$.

\subsection{\boldmath A presentation of $\mm \smallsetminus \Delta_1$\label{sec:u}}
We first claim that a stable curve $C$ of genus $2$
satisfies $[C] \in \mm \smallsetminus \Delta_1$
if and only if
$H^0(\omega_C)$
is basepoint-free.
Indeed, if $p \in C$ is smooth, then by Riemann--Roch for $\oo_C(p)$,
the point $p$ is a basepoint
of $H^0(\omega_C)$ if and only if $\dim H^0(\oo_C(p)) = 2$,
which is impossible since $C$ is not rational.
If $p \in C$ is singular, then write $\tilde{C}$ for the partial normalization
of $C$ at $p$. By Riemann--Roch for $\oo_{\tilde{C}}$,
the point $p$ is a basepoint
of $H^0(\omega_C)$ if and only if $\dim H^0(\oo_{\tilde{C}}(p)) = 2$,
i.e.\ if and only if $p$ is a disconnecting node.
Thus $H^0(\omega_C)$
is basepoint-free if and only if $C$ has no disconnecting nodes, i.e.\ if
and only if $[C] \in \mm \smallsetminus \Delta_1$.

Consequently, every $[C] \in \mm \smallsetminus \Delta_1$
comes equipped with a canonical map to $\pp^1$;
by the Riemann--Hurwitz formula, this map is branched over $6$ points
with multiplicity.
We conclude that the family of curves in weighted projective space given by
\begin{equation} \label{uuniv}
z^2 = ax^6 + bx^5y + cx^4 y^2 + d x^3 y^3 + e x^2 y^4 + f xy^5 + g y^6,
\end{equation}
over the locus of degree $6$ polynomials
$ax^6 + bx^5y + cx^4 y^2 + d x^3 y^3 + e x^2 y^4 + f xy^5 + g y^6$
with no triple root,
contains every isomorphism class $[C] \in \mm \smallsetminus \Delta_1$.

Moreover, since the map to $\pp^1$ is canonical, we conclude
that any isomorphism between two curves in this family
is via a linear change of variables on $x$ and $y$ and scalar multiplication
on $z$, i.e.\ via the action of $\gl_2 \times \mathbb{G}_m$
with $\gl_2$ acting on $(x, y)$ and $\mathbb{G}_m$ acting on $z$.
By inspection, the subgroup of $\gl_2 \times \mathbb{G}_m$
acting trivially is the image of $\mathbb{G}_m$ under the map
$t \mapsto (t \cdot \mathbf{1}, t^3)$.
We conclude that $U$ is isomorphic to the quotient
of the locus of degree $6$ polynomials
$ax^6 + bx^5y + cx^4 y^2 + d x^3 y^3 + e x^2 y^4 + f xy^5 + g y^6$
with no triple root by the natural action of $(\gl_2 \times \mathbb{G}_m) / \mathbb{G}_m$ constructed above.

To recast this in a somewhat nicer form, we observe that
$(\gl_2 \times \mathbb{G}_m) / \mathbb{G}_m$ is itself isomorphic to $\gl_2$
via the map $\gl_2 \times \mathbb{G}_m \to \gl_2$ defined by
\[A \times t \mapsto \frac{\det A}{t} \cdot A,\]
whose kernel is, by inspection, the image of $\mathbb{G}_m$ under the map
$t \mapsto (t \cdot \mathbf{1}, t^3)$.
An inverse to this map is given by
\[A \mapsto A \times \det A.\]
In particular, it follows that
\[\mm \smallsetminus \Delta_1 \simeq 
(V_6(2) \smallsetminus \left\{\text{forms with triple roots}\right\})/\gl_2.\]
Moreover, since $H^0(\omega_C) = \langle x, y \rangle$ with the natural
action of $\gl_2$,
it follows that the Hodge bundle is the pullback of the representation
$V_1$ from $B\!\gl_2$.

\section{Pushforwards along multiplication maps\label{sec:mult}}

Let $X$ be any stack, and
$\mathcal{E} \to X$ be a two-dimensional vector bundle on $X$.
Write $c_i = c_i(\mathcal{E})$, and $x = c_1(\oo_{\pp \mathcal{E}}(1))$
for the tautological class of the projective bundle $\pp \mathcal{E}$.
In this section we give formulas for the pushforward
maps on Chow groups under multiplication maps, similar to those
obtained by Vistoli in \cite{vistoliM2}.

\begin{defi} \label{skr}
For $j \leq r$,
define the class $s^j_r \in \chow^j (\pp \sym^r \mathcal{E})$
as the pushforward of $x_1 x_2 \cdots x_j$
under the multiplication map $(\pp \mathcal{E})^j \times \pp (\sym^{r - j} \mathcal{E}) \to \pp \sym^r \mathcal{E}$,
where $x_i$ denotes the pullback of $x$ under projection to the $i$th
$\pp \mathcal{E}$-factor.
\end{defi}

\begin{lm} \label{rec} The $s^j_r$ can be calculated via the following
recurrence relation:
\[s^0_r = 1 \tand s^{j+1}_r = (t + jc_1) \cdot s^j_r + j(r+1-j) c_2 \cdot s_r^{j-1}\]
where $t$ is the hyperplane class on $\pp \sym^r \mathcal{E}$.
\end{lm}
\begin{proof}
Write $\pi \colon (\pp \mathcal{E})^r \to \pp \sym^r \mathcal{E}$ for the multiplication map.
As $(\pp \mathcal{E})^{r - j} \to \pp \sym^{r - j} \mathcal{E}$ is degree $(r - j)!$,
we obtain
\[s^j_r = \frac{1}{(r - j)!} \pi_*(x_1 \cdots x_j).\]
By push-pull and the symmetry of $\pi$, this implies
\begin{align*}
t \cdot s_r^j &= \frac{1}{(r - j)!} \pi_*\big(x_1 \cdots x_j \cdot (x_1 + \cdots + x_r)\big) \\
&= \frac{j}{(r - j)!} \pi_*(x_1 \cdots x_{j - 1} x_j^2) + \frac{r - j}{(r - j)!} \pi_*(x_1 \cdots x_{j+1}) \\
&= \frac{j}{(r - j)!} \pi_*\big(x_1 \cdots x_{j - 1} (-c_1 x_j - c_2)\big) + \frac{r - j}{(r - j)!} \pi_*(x_1 \cdots x_{j+1}) \\
&= -\frac{jc_1}{(r - j)!} \pi_*(x_1 \cdots x_j) - \frac{j(r + 1 - j)c_2}{(r - j + 1)!} \pi_*(x_1 \cdots x_{j-1}) + \frac{1}{(r - j - 1)!} \pi_*(x_1 \cdots x_{j+1}) \\
&= -jc_1 \cdot s^j_r - j(r + 1 - j)c_2 \cdot s^{j-1}_r + s^{j+1}_r,
\end{align*}
which yields the desired formulas.
\end{proof}

\begin{lm} \label{basismod}
As an $\chow^*(X)$-module, $\chow^* (\sym^r \mathcal{E})$
is generated by the classes $s_r^0, s_r^1, \ldots, s_r^r$.
\end{lm}
\begin{proof}
From the recurrence relations of Lemma~\ref{rec}, it is clear by induction on $j$ that
that $\zz[c_1, c_2]\langle 1, t, t^2, \ldots, t^j\rangle = \zz[c_1, c_2]\langle s_r^0, s_r^1, s_r^2, \ldots, s_r^j \rangle$.
\end{proof}

The point of choosing this system of generators
is that the pushforward along the
multiplication map takes a particularly nice form
in this basis:

\begin{lm} \label{lm:mult}
The pushforward along the multiplication map
\[\pp \sym^a \mathcal{E} \times \pp \sym^b \mathcal{E} \to \pp \sym^{a + b} \mathcal{E}\]
sends
\[s^\alpha_a \times s^\beta_b \mapsto \binom{a - \alpha + b - \beta}{a - \alpha} \cdot s^{\alpha + \beta}_{a + b}.\]
\end{lm}
\begin{proof} Consider the commutative diagram,
all of whose arrows are multiplication maps
(plus permuting the factors):
\[\begin{CD}
(\pp^1)^\alpha \times \pp^{a - \alpha} \times (\pp^1)^\beta \times \pp^{b - \beta} @>>> (\pp^1)^{\alpha + \beta} \times \pp^{a + b - \alpha - \beta} \\
@VVV @VVV \\
\pp^a \times \pp^b @>>> \pp^{a + b}
\end{CD}\]
Then the left and right sides of the desired equality
are the pushforwards around the left and right sides of the diagram
respectively of the class
$x_1 x_2 \cdots x_\alpha \cdot y_1 y_2 \cdots y_\beta$.
\end{proof}

\begin{lm} \label{diag}
The class of the diagonal $\pp \mathcal{E} \subset \pp \mathcal{E} \times \pp \mathcal{E}$
is given by $x_1 + x_2 + c_1$.
\end{lm}
\begin{proof}
By functoriality, and considering the case when $X = B\!\gl_2$ and $\mathcal{E}$ is the
standard representation,
the desired class must be given by
$a x_1 + b x_2 + c c_1$,
for constants $a$, $b$, and $c$.

When $X$ is a point, this formula asserts that the diagonal
in $\pp^1 \times \pp^1$ has class $(a, b)$; thus, $a = b = 1$.
Moreover, when we tensor $\mathcal{E}$ by a line bundle $\mathcal{L}$
with Chern class $\ell = c_1(\mathcal{L})$,
then the class of the diagonal is unaffected, but $x_i$
are replaced by $x_i - \ell$ while $c_1$ is replaced by $c_1 + 2\ell$.
It follows that
\[(x_1 - \ell) + (x_2 - \ell) + c(c_1 + 2\ell) = x_1 + x_2 + c c_1,\]
and so $c = 1$ as desired.
\end{proof}

\begin{lm} \label{diag3}
The class of the triple diagonal $\pp \mathcal{E} \subset \pp \mathcal{E} \times \pp \mathcal{E} \times \pp \mathcal{E}$
is given by
\[(x_1 x_2 + x_2 x_3 + x_3 x_1) + (x_1 + x_2 + x_3) \cdot c_1 + c_1^2 - c_2.\]
\end{lm}
\begin{proof}
By Lemma~\ref{diag}, the image of the diagonal map
$\pp \mathcal{E} \to \pp \mathcal{E} \times \pp \mathcal{E}$
is $x_1 + x_2 + c_1$
where $x_1$ and $x_1$ are the pullbacks of the hyperplane classes
from both factors of $\pp \mathcal{E}$.
Hence, the image of the triple diagonal map
is given by
\begin{align*}
[\Delta_3] &= (x_1 + x_2 + c_1)(x_2 + x_3 + c_1) \\
&= (x_1 x_2 + x_2 x_3 + x_3 x_1) + (x_1 + x_2 + x_3) \cdot c_1 + c_1^2 + x_2^2 + c_1 x_2 \\
&= (x_1 x_2 + x_2 x_3 + x_3 x_1) + (x_1 + x_2 + x_3) \cdot c_1 + c_1^2 - c_2. \qedhere
\end{align*}
\end{proof}

\begin{lm} \label{ver2}
The pushforward along the symmetric square map
\[\pp \mathcal{E} \to \pp \sym^2 \mathcal{E}\]
is given by
\begin{align*}
s_1^0 &\mapsto 2s_2^1 + 2c_1 \\
s_1^1 &\mapsto s_2^2 - 2c_2
\end{align*}
\end{lm}
\begin{proof}
By Lemma~\ref{diag}, the image of the diagonal map
$\pp \mathcal{E} \to \pp \mathcal{E} \times \pp \mathcal{E}$
is
\[[\Delta] = x_1 + x_2 + c_1,\]
where $x_1$ and $x_1$ are the pullbacks of the hyperplane classes
from both factors of $\pp \mathcal{E}$.
Next, we multiply the above expression by $x_1$, to get
\begin{align*}
x_1 \cdot [\Delta] &= x_1 (x_1 + x_2 + c_1) \\
&= (-c_1x_1 - c_2) + x_1(x_2 + c_1) \\
&= x_1 x_2 - c_2.
\end{align*}
The desired formulas follow immediately.
\end{proof}

\begin{lm} \label{ver}
The pushforward along the symmetric cube map
\[\pp \mathcal{E} \to \pp \sym^3 \mathcal{E}\]
is given by
\begin{align*}
s_1^0 &\mapsto 3s_3^2 + 6c_1 s_3^1 + 6(c_1^2 - c_2) \\
s_1^1 &\mapsto s_3^3 - 6c_2 s_3^1 - 6c_1c_2
\end{align*}
\end{lm}
\begin{proof}
By Lemma~\ref{diag3}, the image of the triple diagonal map
$\pp \mathcal{E} \to \pp \mathcal{E} \times \pp \mathcal{E} \times \pp \mathcal{E}$
is given by
\[[\Delta_3] = (x_1 x_2 + x_2 x_3 + x_3 x_1) + (x_1 + x_2 + x_3) \cdot c_1 + c_1^2 - c_2.\]
Next, we multiply the above expression by $x_1$, to get
\begin{align*}
x_1 \cdot [\Delta_3] &= x_1^2 (x_2 + x_3 + c_1) + x_1 x_2 x_3 + x_1 (x_2 + x_3) c_1 + x_1(c_1^2 - c_2) \\
&= (-c_1 x_1 - c_2) (x_2 + x_3 + c_1) + x_1 x_2 x_3 + x_1 (x_2 + x_3) c_1 + x_1(c_1^2 - c_2) \\
&= x_1 x_2 x_3 - (x_1 + x_2 + x_3) \cdot c_2 - c_1 \cdot c_2.
\end{align*}
The desired formulas follow immediately.
\end{proof}

\begin{lm} \label{lm:sub} Let $\mathcal{V} \subset \mathcal{W}$ be an inclusion of vector bundles
over $X$. Then the fundamental class of $\pp \mathcal{V} \subset \pp \mathcal{W}$
in $\chow^*(\pp \mathcal{W})$ is given by the Chern polynomial
of the quotient bundle $\mathcal{W} / \mathcal{V}$.
That is, if $\dim \mathcal{W}/\mathcal{V} = d$, then
\[[\pp \mathcal{V}] = x^d + c_1(\mathcal{W} / \mathcal{V}) x^{d-1} + \dots + c_d(\mathcal{W} / \mathcal{V}),\]
where $x = c_1( \mathcal{O}_{\pp \mathcal{W}}(1))$.
\end{lm}
\begin{proof}
The composition $\mathcal{O}_{\pp \mathcal{W}}(-1) \to \pi^*\mathcal{W} \to \pi^*(\mathcal{W}/\mathcal{V})$,
where $\pi \colon \mathcal{W} \to X$ denotes the structure map,
gives a section of $\pi^*(\mathcal{W} / \mathcal{V}) \otimes \mathcal{O}_{\pp \mathcal{W}}(1)$
whose vanishing locus is
$\pp \mathcal{V} \subset \pp \mathcal{W}$.
Thus
\[[\pp \mathcal{V}] = c_d(\pi^*(\mathcal{W} / \mathcal{V}) \otimes \mathcal{O}_{\pp \mathcal{W}}(1)) = x^d + c_1(\mathcal{W} / \mathcal{V}) x^{d-1} + \dots + c_d(\mathcal{W} / \mathcal{V}). \qedhere\]
\end{proof}

\begin{lm} \label{lm:segre}
For rank~$2$ vector bundles $\mathcal{E}_1$ and $\mathcal{E}_2$ on $X$,
the pushforward along the Segre map
\[\pp \mathcal{E}_1 \times \pp \mathcal{E}_2 \to \pp(\mathcal{E}_1 \otimes \mathcal{E}_2)\]
is given by
\begin{align*}
1 &\mapsto 2x + c_1(\mathcal{E}_1) + c_1(\mathcal{E}_2) \\
x_1 &\mapsto x^2 + c_1(\mathcal{E}_2) x + c_2(\mathcal{E}_2) - c_2(\mathcal{E}_1) \\
x_2 &\mapsto x^2 + c_1(\mathcal{E}_1) x + c_2(\mathcal{E}_1) - c_2(\mathcal{E}_2) \\
x_1 x_2 &\mapsto x^3 + (c_1(\mathcal{E}_1) + c_1(\mathcal{E}_2)) x^2 + (c_2(\mathcal{E}_1) + c_1(\mathcal{E}_1) c_1(\mathcal{E}_2) + c_2(\mathcal{E}_2)) x \\
&\qquad + c_1(\mathcal{E}_1) c_2(\mathcal{E}_2) + c_2(\mathcal{E}_1) c_1(\mathcal{E}_2).
\end{align*}
\end{lm} 
\begin{proof} 
To find the pushforward of $1$, we note that by functoriality,
the desired class must be given by $ax + b c_1(\mathcal{E}_1) + c c_1(\mathcal{E}_2)$,
for constants $a$ and $b$.
When $X$ is a point, this formula asserts that the Segre surface $\pp^1 \times \pp^1 \subset \pp^3$
is of degree $a$; thus $a = 2$.
Moreover, when we tensor $\mathcal{E}_1$ by a line bundle $\mathcal{L}$
with Chern class $\ell = c_1(\mathcal{L})$,
then the pushforward is unaffected, but $x$ is replaced by
$x - \ell$ while $c_1(\mathcal{E}_1)$ is replaced by $c_1(\mathcal{E}_1) + 2\ell$
and $c_1(\mathcal{E}_2)$ is unchanged.
It follows that
\[2(x - \ell) + b (c_1(\mathcal{E}_1) + 2\ell) + c c_1(\mathcal{E}_2) = 2x + b c_1(\mathcal{E}_1) + c c_1(\mathcal{E}_2),\]
and so $b = 1$ as desired. By symmetry, $c = 1$ too.

To find the pushforward of $x_1$, we note that by functoriality,
the desired class must be given by
\[a x^2 + b c_1(\mathcal{E}_1) x  + c c_1(\mathcal{E}_2) x + d c_2(\mathcal{E}_1) + e c_2(\mathcal{E}_2) + f c_1(\mathcal{E}_1)^2 + g c_1(\mathcal{E}_2)^2 + h c_1(\mathcal{E}_2) c_1(\mathcal{E}_1),\]
for constants $a$, $b$, $c$, $d$, $e$, $f$, $g$, and $h$.
Moreover, by symmetry, the pushforward of $x_2$ must be given by
\[a x^2 + c c_1(\mathcal{E}_1) x  + b c_1(\mathcal{E}_2) x + e c_2(\mathcal{E}_1) + d c_2(\mathcal{E}_2) + g c_1(\mathcal{E}_1)^2 + f c_1(\mathcal{E}_2)^2 + h c_1(\mathcal{E}_2) c_1(\mathcal{E}_1).\]
In particular, the pushforward of $x_1 + x_2$ must be given by
\begin{multline*}
2ax^2 + (b + c)(c_1(\mathcal{E}_1) + c_1(\mathcal{E}_2)) x + (d + e)(c_2(\mathcal{E}_1) + c_2(\mathcal{E}_2)) \\
+ (f + g)(c_1(\mathcal{E}_1)^2 + c_2(\mathcal{E}_2)^2) + 2h c_1(\mathcal{E}_2) c_1(\mathcal{E}_1).
\end{multline*}
But by push-pull and our previous calculation of the pushforward of $1$,
the pushforward of $x_1 + x_2$ must also be
\[x \cdot (2x + c_1(\mathcal{E}_1) + c_1(\mathcal{E}_2)) = 2x^2 + x(c_1(\mathcal{E}_1) + c_1(\mathcal{E}_2)).\]
We conclude that $a = b + c = 1$, and $d + e = f + g = h = 0$.
In other words, the pushforward of $x_1$ must be given by
\[x^2 + b c_1(\mathcal{E}_1) x  + (1 - b) c_1(\mathcal{E}_2) x + d c_2(\mathcal{E}_1) - d c_2(\mathcal{E}_2) + f c_1(\mathcal{E}_1)^2 - f c_1(\mathcal{E}_2)^2\]
for constants $b$, $d$, and $f$.
When $\mathcal{E}_1$ is trivial, the pushforward of $x_1$ is simply the fundamental class
of the diagonal $\pp \mathcal{E}_2 \subset \pp(\mathcal{E}_2 \oplus \mathcal{E}_2)$,
so by Lemma~\ref{lm:sub}
\[x^2 + (1 - b) c_1(\mathcal{E}_2) x - d c_2(\mathcal{E}_2) - f c_1(\mathcal{E}_2)^2 = x^2 + c_1(\mathcal{E}_2) x + c_2(\mathcal{E}_2).\]
Thus $d = -1$ and $b = f = 0$, and so the class of the pushforward of $x_1$ is given by
\[x^2 + c_1(\mathcal{E}_2) x + c_2(\mathcal{E}_2) - c_2(\mathcal{E}_1).\]
And by symmetry, the class of the pushforward of $x_2$ is given by
\[x^2 + c_1(\mathcal{E}_1) x + c_2(\mathcal{E}_1) - c_2(\mathcal{E}_2).\]

Finally, to calculate the pushforward of $x_1 x_2$,
we note that by the projection formula, the pushforward of $(x_1 + x_2) \cdot x_2$ is
\[x \cdot (x^2 + c_1(\mathcal{E}_1)x + c_2(\mathcal{E}_1) - c_2(\mathcal{E}_2)).\]
On the other hand,
$x_2^2 = - c_1(\mathcal{E}_2) x_2 - c_2(\mathcal{E}_2)$, so the pushforward
of $x_2^2$ is
\[- c_1(\mathcal{E}_2) \cdot (x^2 + c_1(\mathcal{E}_1) x + c_2(\mathcal{E}_1) - c_2(\mathcal{E}_2)) - c_2(\mathcal{E}_2) \cdot (2x + c_1(\mathcal{E}_1) + c_1(\mathcal{E}_2)).\]
Subtracting, we obtain that the pushforward of $x_1 x_2$ is
\[x^3 + (c_1(\mathcal{E}_1) + c_1(\mathcal{E}_2)) x^2 + (c_2(\mathcal{E}_1) + c_1(\mathcal{E}_1) c_1(\mathcal{E}_2) + c_2(\mathcal{E}_2)) x + c_1(\mathcal{E}_1) c_2(\mathcal{E}_2) + c_2(\mathcal{E}_1) c_1(\mathcal{E}_2)\]
as desired.
\end{proof}

\section{The Chow ring of $BG$\label{sec:bg}}
Let $G  = (\mathbb{G}_m \times \mathbb{G}_m)
\rtimes \mathbb{Z}/2$ as in Definition~\ref{G}.
Our goal in this section is to compute the Chow ring of $BG$.
These results (as well as the results we shall obtain
in Section~\ref{sec:pushpull}) are related,
at least in some sense, to results on wreath
products with cyclic groups obtained by Totaro (c.f.\ Section~2.8 of~\cite{totaro}).

To start with, we observe that we have a natural embedding
$G \hookrightarrow \gl_2$ as the stabilizer
of an unordered pair of distinct lines in $V$.
In particular, $BG$
is identified with the complement of the image
of the squaring (Veronese) map
\[\pp V_1 / \gl_2 \to \pp V_2 / \gl_2.\]
Write $t = c_1(\oo_{\pp V_2}(1))$.

\begin{lm} The Chow ring of $\pp V_2 / \gl_2$
is given by
\[\chow^* (\pp V_2 / \gl_2) \simeq \zz[\alpha_1, \alpha_2, t] / ((t^2 + 2\alpha_1 t + 4 \alpha_2)(t + \alpha_1)).\]
\end{lm}
\begin{proof}
Write $a_1$ and $a_2$ for the Chern roots of $V$.
From Grothendieck's projective bundle formula, the Chow ring
of $\pp V_2$ is given by
\[\zz[\alpha_1, \alpha_2, t] / p,\]
where
\[p = (t - 2a_1)(t - 2a_2)(t - a_1 - a_2) = (t^2 - 2\alpha_1 t + 4 \alpha_2)(t - \alpha_1). \qedhere\]
\end{proof}
	
From Lemma~\ref{ver2} and the localization exact sequence for Chow
rings, we conclude that $\chow^*(BG)$ is the quotient of $\chow^* (\pp V_2/ \gl_2)$
by the relations $2s_2^1 - 2\alpha_1 = s_2^2 - 2\alpha_2 = 0$.
Using Lemma~\ref{rec}, we calculate
\begin{align*}
s_2^0 &= 1 \\
s_2^1 &= t \\
s_2^2 &= t^2 - \alpha_1 t + 2 \alpha_2.
\end{align*}
Thus
\begin{align*}
2s_2^1 - 2\alpha_1 &= 2t - 2\alpha_1 \\
s_2^2 - 2\alpha_2 &= t^2 - \alpha_1 t.
\end{align*}
Note that these relations imply $(t^2 - 2\alpha_1 t + 4 \alpha_2)(t - \alpha_1) = 0$.
We conclude that
\begin{equation} \label{bgt}
\chow^*(BG) \simeq \zz[\alpha_1, \alpha_2, t] / (2t - 2\alpha_1, t^2 - \alpha_1 t).
\end{equation}

\begin{thm} \label{thm:bg} The Chow ring of $BG$ is given by
\[\chow^*(BG) \simeq \zz[\beta_1, \beta_2, \gamma] / (2\gamma, \gamma^2 + \beta_1\gamma).\]
\end{thm}
\begin{proof}
We have already done most of the work; it remains just to note that
the representation $V$ of $\gl_2$ restricts to the representation $W$ of $G$,
so $\alpha_i = \beta_i$, and to
identify $\gamma$
and rewrite the presentation \eqref{bgt} in terms of $\gamma$.

Because $\gamma$ is the pullback to $\chow^1(BG)$
of a nontrivial $2$-torsion element in $\chow^1(B(\zz/2\zz))$,
and $G$ splits as a semidirect product, $\gamma$ is a nontrivial $2$-torsion element
in $\chow^1(BG)$.
From \eqref{bgt}, the only such element is $t - \alpha_1$. Substituting $t = \gamma + \alpha_1$
into \eqref{bgt} yields the desired result.
\end{proof}

\section{\label{genrel} Generators and some relations}

In this section, we show that $\chow^*(\mm)$ is generated by $\lambda_1$, $\lambda_2$, and $\delta_1$,
as per Theorem~\ref{thm:main}.
We then establish several relations that these generators satisfy;
in characteristic zero, this includes all relations claimed in Theorem~\ref{thm:main}.

\begin{lm} \label{delta1delta1}
The restriction $\delta_1|_{\Delta_1} = \gamma - \lambda_1$.
\end{lm}
\begin{proof}
The normal bundle of $\Delta_1$ in $\bar{M}_2$
is given by the line bundle whose fiber over $E_1 \cup_p E_2$
is canonically identified with $T_p E_1 \otimes T_p E_2$.
By inspection this differs from
$\wedge^2 H^0(\omega_{E_1 \cup_p E_2})^\vee$ by the class $\gamma$,
which gives the desired formula.
\end{proof}

From Theorem~\ref{thm:bg} and the presentation given in Section~\ref{sec:s},
we see that the Chow ring of $\Delta_1$ is generated by $\lambda_1$, $\lambda_2$, and $\gamma$;
similarly, from the presentation given in Section~\ref{sec:u},
we see that the Chow ring of $\mm \smallsetminus \Delta_1$ is generated by $\lambda_1$ and $\lambda_2$.
Applying the localization sequence
\[\chow^{* - 1}(\Delta_1) \to \chow^*(\mm) \to\chow^*(\mm \smallsetminus \Delta_1) \to 0,\]
together with Lemma~\ref{delta1delta1}, we see that $\chow^*(\mm)$
is generated by $\lambda_1$, $\lambda_2$, and $\delta_1$ as promised.

\subsection{\boldmath Relations from $\chow^*(BG)$}
Since the relations $2\gamma = \gamma^2 + \lambda_1 \gamma = 0$ hold in $\chow^*(BG)$,
they must also hold in $\chow^*(\Delta_1)$. Using Lemma~\ref{delta1delta1},
we may write these relations as
\[2\delta_1 + 2\lambda_1 = \delta_1^2 + \delta_1\lambda_1 = 0 \in \chow^*(\Delta_1).\]
Pushing these relations forward to $\mm$, we obtain relations in $\chow^*(\mm)$:
\[2\delta_1^2 + 2\delta_1\lambda_1 = \delta_1^3 + \delta_1^2\lambda_1 = 0 \in \chow^*(\mm).\]

\subsection{Relations from Grothendieck--Riemann--Roch\label{relgrr}}
Here we recall Mumford's proof of several
relations in $\chow^*(\mm)$ in \cite{enum}, taking care
to work with integral coefficients rather than rational coefficients
where possible.

Let $\pi \colon \mathcal{C} \to \mm$ be the universal curve.
Write $\omega_\pi$ for the relative dualizing sheaf, and $\Omega_\pi^1$
for the sheaf of relative differentials; these are sheaves on $\mathcal{C}$.
Applying the Grothendieck--Riemann--Roch theorem to
calculate the pushforward of the structure sheaf of $\mathcal{C}$
under $\pi$, we obtain:

\begin{lm} \label{lm:grr}
The following relations hold modulo torsion:
\begin{align*}
2 &= \pi_*(c_1(\Omega^1_\pi)) \\
12 \lambda_1 &= \pi_*(c_1(\Omega^1_\pi)^2 + c_2(\Omega^1_\pi)) \\
12 \lambda_1^2 - 24\lambda_2 &= \pi_*(c_1(\Omega^1_\pi)c_2(\Omega^1_\pi)).
\end{align*}
\end{lm}
\begin{rem} \label{rem:pappas}
Pappas has shown that, for maps of nonnegative relative dimension in characteristic zero,
the relations given by the Grothendieck--Riemann--Roch theorem
hold integrally when denominators are cleared \cite{pappas}.
Consequently these relations hold exactly in characteristic zero.
\end{rem}

Write $\mathcal{S} = \mathcal{S}_0 \cup \mathcal{S}_1 \subset \mathcal{C}$
for the universal singular locus, where $\mathcal{S}_0$ denotes the universal
self node and $\mathcal{S}_1$ denotes the universal disconnecting node.
Since $\Omega^1_\pi \simeq \omega_\pi \otimes \mathcal{I}_{\mathcal{S}}$,
we have
\[c_1(\Omega^1_\pi) = c_1(\omega_\pi) \tand c_2(\Omega^1_\pi) = [\mathcal{S}].\]
The relations of Lemma~\ref{lm:grr} may thus be written as
\begin{align*}
2 &= \pi_*(c_1(\omega_\pi)) \\
12 \lambda_1 &= \pi_*(c_1(\omega_\pi)^2 + [\mathcal{S}]) \\
12 \lambda_1^2 - 24\lambda_2 &= \pi_*(c_1(\omega_\pi) \cdot [\mathcal{S}]).
\end{align*}

\begin{lm} \label{kappa} We have the following relation in $\chow^*(\mathcal{C})$:
\[c_1(\omega_\pi)^2 - c_1(\omega_\pi) \lambda_1 + \lambda_2 - [\mathcal{S}_1] = 0.\]
\end{lm}
\begin{proof}
As shown in Section~\ref{sec:u}, the natural map $\pi^* \pi_* \omega_\pi \to \omega_\pi$
vanishes exactly along $\mathcal{S}_1$.
Let $\mathcal{L}$ denote its kernel, so we have an exact sequence:
\[0 \to \mathcal{L} \to \pi^* \pi_* \omega_\pi \to \omega_\pi \otimes \mathcal{I}_{\mathcal{S}_1} \to 0.\]
Since $\mathcal{S}_1$ is lci of codimension $2$,
the projective dimension of $\omega_\pi \otimes \mathcal{I}_{\mathcal{S}_1}$ is $1$;
consequently $\mathcal{L}$ is a line bundle, and so
\begin{align*}
0 = c_2(\mathcal{L}) &= \left[\frac{1 + c_1(\pi^* \pi_* \omega_\pi) + c_2(\pi^* \pi_* \omega_\pi)}{1 + c_1(\omega_\pi \otimes \mathcal{I}_{\mathcal{S}_1}) + c_2(\omega_\pi \otimes \mathcal{I}_{\mathcal{S}_1})}\right]_2 = \left[\frac{1 + \lambda_1 + \lambda_2}{1 + c_1(\omega_\pi) + [\mathcal{S}_1]}\right]_2 \\
&= c_1(\omega_\pi)^2 - c_1(\omega_\pi) \lambda_1 + \lambda_2 - [\mathcal{S}_1]. \qedhere
\end{align*}
\end{proof}

\begin{lm} \label{lm2448} We have the following relation in $\chow^*(\mm)$ modulo torsion:
\[24 \lambda_1^2 - 48 \lambda_2 = 0.\]
\end{lm}
\begin{rem} In characteristic zero, the arguments of this section establish this relation exactly
(not just modulo torsion);
c.f.\ Remark~\ref{rem:pappas}.
\end{rem}
\begin{proof} In light of the relation $12 \lambda_1^2 - 24\lambda_2 = \pi_*(c_1(\omega_\pi) \cdot [\mathcal{S}])$,
it remains to show $c_1(\omega_\pi) \cdot [\mathcal{S}]$
is $2$-torsion.
But this is clear since the residue map gives a trivialization
of $\omega_\pi|_\mathcal{S}$ up to sign.
\end{proof}

\begin{lm} \label{delta0} We have the following relation in $\chow^*(\mm)$ modulo torsion:
\[\delta_0 = 10 \lambda_1 - 2 \delta_1.\]
\end{lm}
\begin{rem} In characteristic zero, the arguments of this section establish this relation exactly
(not just modulo torsion);
c.f.\ Remark~\ref{rem:pappas}.
Combined with Lemma~\ref{delta0lambda2} below, we obtain the relation
$20 \lambda_1 \lambda_2 - 4 \delta_1 \lambda_2 = 0$.
\end{rem}
\begin{proof}
Combining the relations
$2 = \pi_*(c_1(\omega_\pi))$ and 
$12 \lambda_1 = \pi_*(c_1(\omega_\pi)^2 + [\mathcal{S}])$ with Lemma~\ref{kappa}, we obtain
\begin{align*}
12 \lambda_1 &= \pi_*(c_1(\omega_\pi)^2 + [\mathcal{S}]) \\
&= \pi_*(c_1(\omega_\pi) \lambda_1 - \lambda_2 + [\mathcal{S}_1] + [\mathcal{S}_0] + [\mathcal{S}_1]) \\
&= \pi_*(c_1(\omega_\pi)) \lambda_1 + 2 \pi_*([\mathcal{S}_1]) + \pi_*(\mathcal{S}_0) \\
&= 2 \lambda_1 + 2 \delta_1 + \delta_0.
\end{align*}
This yields the desired relation upon rearrangement.
\end{proof}

\begin{lm} \label{delta0lambda2}
We have the following relation in $\chow^*(\mm)$:
\[2\delta_0 \lambda_2 = 0.\]
\end{lm}
\begin{proof}
The map $\bar{M}_{1,2} \to \mm$
obtained by gluing together
the marked points is $2$-to-$1$ onto $\Delta_0$,
so the composition of pullback followed by pushforward
$\chow^*(\bar{M}_2) \to \chow^*(\bar{M}_{1,2}) \to \chow^*(\bar{M}_2)$
is multiplication by $2\delta_0$.

It thus suffices to show
that the pullback of $\lambda_2$ to $\bar{M}_{1,2}$ is zero.
But this is clear since the residue map gives a surjection
from the pullback of the Hodge bundle to the structure sheaf of
$\bar{M}_{1,2}$.
\end{proof}
 
\section{\label{sec:pushpull} Pushforward and pullback along $B(\mathbb{G}_m \times \mathbb{G}_m) \to BG$}

In this section, we calculate the pushforward and pullback maps
along
\[\pi \colon B(\mathbb{G}_m \times \mathbb{G}_m) \to BG.\]
Write $t_1$ and $t_2$ in $\chow^*(\mathbb{G}_m \times \mathbb{G}_m)$
for the Chern classes of the standard representations
of each $\mathbb{G}_m$ factor. Since pullback is induced by
restrictions of representations, we immediately find
\begin{align*}
\pi^*(\beta_1) &= t_1 + t_2 \\
\pi^*(\beta_2) &= t_1 t_2 \\
\pi^*(\gamma) &= 0.
\end{align*}

The question of pushforward is more subtle, and we will deduce
it from a special case of an integral Grothendieck--Riemann--Roch theorem.

Let $\pi \colon X \to Y$ be finite flat map ramified along a divisor $R$,
and $\mathcal{F}$ be an element of the Grothendieck group of sheaves on $X$.
Then a special case of the ordinary Grothendieck--Riemann--Roch
theorem with rational coefficients states that
\[c_1(\pi_* \mathcal{F}) = \pi_*\left(c_1(\mathcal{F}) - \frac{\rk \mathcal{F}}{2} \cdot R\right) \mod \text{torsion}.\]

For maps of nonnegative relative dimension in characteristic zero,
Pappas has shown in \cite{pappas} that the
Grothendieck--Riemann--Roch theorem holds integrally when we clear denominators;
equivalently, in this special case,
the above identity holds integrally when $\mathcal{F}$
is a multiple of~$2$ in the Grothendieck group of sheaves (which forces $\rk \mathcal{F}$ to be even).

Here we note that, in this special case, this identity holds integrally in arbitrary characteristic,
subject only to the condition that $\rk \mathcal{F}$ is even:

\begin{lm} \label{grr1} If $\rk \mathcal{F}$ is even, then
\[c_1(\pi_* \mathcal{F}) = \pi_*\left(c_1(\mathcal{F}) - \frac{\rk \mathcal{F}}{2} \cdot R\right).\]
\end{lm}
\begin{proof}
Resolving our sheaves, we reduce to the case when
$\mathcal{F}$ is a linear combination of vector bundles.
Applying the splitting principle, we further reduce to the case when $\mathcal{F}$
is a linear combination of line bundles.
Pulling back to the total space of the dual of these line bundles, we further reduce to the case
when the duals of these line bundles are effective.
Finally, using the relation
$[\oo(-D)] = [\oo] - [\oo_D]$,
we reduce to the case when $\mathcal{F}$ is a linear combination of a trivial bundle $\mathcal{O} \oplus \mathcal{O}$
of rank $2$ and structure sheaves of divisors $\oo_D$.
It thus remains to observe that the desired result holds in
these two cases, i.e.\ that
\[2 c_1(\pi_* \oo) = -\pi_*(R) \tand c_1(\pi_* \oo_D) = \pi_*(c_1(\oo_D)).\]

The first of these equalities can be seen by considering the ``map''
$\pi_* \oo \to \oo^{\deg \pi}$ given by evaluation at the points of the fiber;
this ``map'' is defined up to the action of the symmetric group on the fiber,
so the square of its determinant is well-defined and gives a map
$(\wedge^{\deg \pi} \pi_* \oo)^{\otimes 2} \to \oo$, which by inspection
vanishes exactly along $\pi_* R$.

The second of these equalities can be seen by
observing that both sides equal $\pi_* D$.
\end{proof}

\begin{cor} \label{pushc1}
If $\mathcal{L}$ is a line bundle on $X$, then
\[\pi_* (c_1(\mathcal{L})) = c_1(\pi_* \mathcal{L}) - c_1(\pi_* \oo).\]
\end{cor}
\begin{proof}
This follows from Lemma~\ref{grr1} applied to $[\mathcal{L}] - [\mathcal{O}]$ (which is of even rank).
\end{proof}

\begin{lm} \label{pushbg} The pushforward map 
along $\pi \colon B(\mathbb{G}_m \times \mathbb{G}_m) \to BG$
may be described recursively as follows:
\begin{align*}
\pi_*(1) &= 2 \\
\pi_*(t_1) &= \beta_1 + \gamma \\
\pi_*(t_1^a) &= \beta_1 \pi_*(t_1^{a - 1}) - \beta_2 \pi_*(t_1^{a - 2}) \quad \text{for $a \geq 2$}\\
\pi_*(t_1^a t_2^b) &= \beta_2^{\min(a, b)} \pi_*(t_1^{|a - b|}).
\end{align*}
\end{lm}
\begin{proof}
Since $\pi$ is of degree $2$, we have $\pi_*(1) = 2$.

From Corollary~\ref{pushc1}, we see that the pushforward of $t_1$
is given by the difference between the induced representations
from $\mathbb{G}_m \times \mathbb{G}_m$ to $BG$
of the standard representation of the first $\mathbb{G}_m$-factor
and the trivial representation. These representations are
the standard representation of $G$ via its inclusion in $\gl_2$
as the normalizer of a maximal torus, respectively the sign representation
of the $\zz/2\zz$ quotient of $G$. Since these have Chern classes $\beta_1$
and $\gamma$ respectively (and $2\gamma = 0$), we obtain
$\pi_*(t_1) = \beta_1 - \gamma = \beta_1 + \gamma$ as desired.

\smallskip

\noindent
Using push-pull, we have for $a \geq 2$:
\[\pi_*(t_1^a) = \pi_*((t_1 + t_2) t_1^{a - 1} - t_1t_2 t_1^{a - 2}) = \beta_1 \pi_*(t_1^{a - 1}) - \beta_2 \pi_*(t_1^{a - 2}).\]
And finally, using push-pull and symmetry, we have
\[\pi_*(t_1^a t_2^b) = \pi_*((t_1 t_2)^{\min(a, b)} t_1^{|a - b|}) = \beta_2^{\min(a, b)} \pi_*(t_1^{|a - b|}). \qedhere\]
\end{proof}

We conclude this section by calculating the Chern classes of the representations
$V_n \bpls V_n$, which are obtained by pushing forward the representation $V_n$
of $\mathbb{G}_m \times \mathbb{G}_m$ with weight $(n, 0)$ to $BG$:

\begin{lm} \label{npower}
The Chern classes of $W_n$ are given by:
\begin{align*}
c_1(W_n) &= n \beta_1 + (n + 1) \gamma \\
c_2(W_n) &= n^2 \beta_2.
\end{align*}
\end{lm}
\begin{proof}
We argue by induction on $n$. When $n = 0$, the representation $W_0$
is the regular representation of the $\zz/2\zz$ quotient of $G$,
which splits as $\mathbf{1} \oplus \Gamma$;
its Chern classes are thus $\gamma$ and $0$ respectively.
When $n = 1$, the representation $W_1 = W$ has Chern classes $\beta_1$
and $\beta_2$ by definition.

For the inductive hypothesis, we suppose $n \geq 2$. Observe that we have
a direct sum decomposition
\[W_{n-1} \otimes W_1 \simeq W_n \oplus (W_{n - 2} \otimes \wedge^2 W_1 \otimes \Gamma).\]
Writing $a_n$ and $b_n$ for the Chern roots of $W_n$, we obtain
\begin{multline} \label{blah}
(1 + a_{n-1} + a_1)(1 + a_{n-1} + b_1)(1 + b_{n-1} + a_1)(1 + b_{n-1} + b_1) \\
= (1 + a_n)(1 + b_n)(1 + a_{n-2} + \beta_1 + \gamma)(1 + b_{n - 2} + \beta_1 + \gamma).
\end{multline}
Comparing terms of degree $1$ on both sides of \eqref{blah}, and using the relation $2\gamma = 0$, we have
\[2 (a_{n - 1} + b_{n - 1}) + 2(a_1 + b_1) = a_n + b_n + a_{n - 2} + b_{n - 2} + 2\beta_1,\]
and so by our inductive hypothesis
\begin{align*}
c_1(W_n) = a_n + b_n &= 2 (a_{n - 1} + b_{n - 1}) + 2(a_1 + b_1) - (a_{n - 2} + b_{n - 2}) - 2\beta_1 \\
&= 2 [(n - 1) \beta_1 + n \gamma] + 2\beta_1 - [(n - 2) \beta_1 + (n - 1) \gamma] - 2 \beta_1 \\
&= n\beta_1 + (n + 1) \gamma.
\end{align*}
Similarly, comparing terms of degree $2$ on both sides of \eqref{blah}, and using $2\gamma = 0$, we have
\begin{multline*}
(a_{n-1} + b_{n-1})^2 + 2a_{n-1}b_{n-1} + (a_1 + b_1)^2 + 2a_1b_1 + 3(a_{n-1} + b_{n-1})(a_1 + b_1) \\
= a_n b_n + a_{n-2} b_{n-2} + (\beta_1 + \gamma)(a_{n - 2} + b_{n - 2}) + (\beta_1 + \gamma)^2 + (a_n + b_n)(a_{n - 2} + b_{n - 2} + 2\beta_1),
\end{multline*}
and so by our inductive hypothesis
\begin{align*}
c_2(W_n) &= a_n b_n \\
&= (a_{n-1} + b_{n-1})^2 + 2a_{n-1}b_{n-1} + (a_1 + b_1)^2 + 2a_1b_1 \\
&\qquad + 3(a_{n-1} + b_{n-1})(a_1 + b_1) - a_{n-2} b_{n-2} - (\beta_1 + \gamma)(a_{n - 2} + b_{n - 2}) \\
&\qquad - (\beta_1 + \gamma)^2 - (a_n + b_n)(a_{n - 2} + b_{n - 2} + 2\beta_1) \\
&= [(n - 1) \beta_1 + n \gamma]^2 + 2 (n - 1)^2 \beta_2 + \beta_1^2 + 2 \beta_2 \\
&\qquad + 3 [(n - 1) \beta_1 + n \gamma] \beta_1 - (n - 2)^2 \beta_2 - (\beta_1 + \gamma) [(n - 2) \beta_1 + (n - 1) \gamma] \\
&\qquad - (\beta_1 + \gamma)^2 - [n \beta_1 + (n + 1) \gamma]([(n - 2) \beta_1 + (n - 1) \gamma] + 2\beta_1) \\
&= n^2 \beta_2 - (n - 1) (\gamma^2 + \beta_1 \gamma) \\
&= n^2 \beta_2. \qedhere
\end{align*}
\end{proof}

\section{The Chow ring of $\Delta_1$\label{sec:adelta1}}

In this section, we calculate the Chow ring $\chow^*(\Delta_1)$.
Recall the presentation of $\Delta_1$ obtained in Section~\ref{sec:s}:
\[\Delta_1 \simeq \left(W_{4,6} \smallsetminus (L_{4,6} \times 0 \cup 0 \times L_{4,6})\right)/G.\]
In this section, we will use the calculations of $\chow^*(BG) \simeq \chow^*(W_{4,6}/G)$
from Section~\ref{sec:bg} to compute the Chow ring of $\Delta_1$ using the localization sequence.

We first note that the formulas in Section~\ref{sec:mult} apply to calculate
loci in the total space of vector bundles with the origin excised: Simply pull back
the corresponding classes from the projectivization.
So to start with, we excise the origin of $W_{4,6}$ using the localization sequence
\[\chow^{* - 4}(BG) \to \chow^*(W_{4,6}/G) \to \chow^*(W_{4,6} \smallsetminus (0 \times 0)) \to 0.\]
In terms of the isomorphism $\chow^*(BG) \simeq \chow^*(W_{4,6}/G)$,
this first map is multiplication by the Euler class $e(W_{4,6}) = c_4(W_{4,6})$, and so
\[\chow^*(W_{4,6} \smallsetminus (0 \times 0)) \simeq \chow^*(BG) / c_4(W_{4,6}).\]
To compute $c_4(W_{4,6})$, we use Lemma~\ref{npower}, which gives
\begin{equation} e(W_{4,6}) = c_4(W_{4,6}) = 
c_2(W_4) \cdot c_2(W_6) = 
 16 \lambda_2 \cdot 36 \lambda_2 = 576 \lambda_2^2.  \end{equation}
To get from $W_{4,6} \smallsetminus (0 \times 0)$ to $\Delta_1$,
we have to excise 
the closed substack
\[Z = \Big(\big((L_{4,6} \smallsetminus 0) \times 0\big) \cup \big(0 \times (L_{4,6} \smallsetminus 0)\big)\Big) / G \simeq \big((L_{4,6} \smallsetminus 0) \times 0\big) / (\mathbb{G}_m \times \mathbb{G}_m).\]
This isomorphism implies that
$\chow^*(Z)$ is generated by $\left\{1, t_1\right\}$ as a module over
$\chow^*(BG)$,
so it suffices to determine the images of $1$ and $t_1$ under
the localization sequence
\[\chow^{* - 2}(Z) \to \chow^*(W_{4,6} \smallsetminus (0 \times 0)) \to \chow^*(\Delta_1) \to 0.\]
But from Lemma~\ref{lm:sub}, the class of
\[\big((L_{4,6} \smallsetminus 0) \times 0\big)/(\mathbb{G}_m \times \mathbb{G}_m) \subset \big(L_{4,6} \oplus L_{4,6} \smallsetminus (0 \times 0)\big) / (\mathbb{G}_m \times \mathbb{G}_m)\]
is given by $c_2(0 \times L_{4,6}) = 24 t_2^2$.
Consequently, the images of $1$ and $t_1$ are the pushforwards of
$24 t_2^2$ and $24 t_1 t_2^2$ along $B(\mathbb{G}_m \times \mathbb{G}_m) \to BG$;
by Lemma~\ref{pushbg}, these are $24 \lambda_1^2 - 48 \lambda_2$ and $24 \lambda_1 \lambda_2$
respectively.
These imply our earlier relation $576 \lambda_2^2 = 0$, and so it follows that:

\begin{thm} \label{adelta1} The Chow ring of $\Delta_1$ is given by
\begin{equation} 
\chow^*(\Delta_1) = \zz[\lambda_1, \lambda_2, \gamma]/(2\gamma, \gamma^2 +
\lambda_1 \gamma, 24 \lambda_1^2 - 48 \lambda_2, 24 \lambda_1 \lambda_2).
\end{equation} 
\end{thm}

\begin{rem}
One observes from this that after inverting 2, the Chow ring is precisely
$\sym^2 \chow^*(\mone)$.
\end{rem}

\section{\label{sec:au} The Chow Ring and Higher Chow Groups of $\mm \smallsetminus \Delta_1$}

In this section we compute the Chow ring and higher Chow groups with $\ell$-adic coefficients
(where $\ell$ is prime to the characteristic of the base field $k$),
of $\mm \smallsetminus \Delta_1$
and of its quotient by the \emph{twisting action} of $\mathbb{G}_m$ --- i.e.\ the action
of $\mathbb{G}_m$ on $\mm \smallsetminus \Delta_1$ obtained by scaling with weight $1$
the coefficients $a$, $b$, $c$, $d$, $e$, $f$, and $g$ in \eqref{uuniv}.

Bloch's higher Chow groups \cite{bloch} are defined as the homology of certain complexes
$z^*(X, \bullet)$
(here we have one complex for each value of $*$, and $\bullet$ denotes the grading of the complexes).
These groups are
significantly easier to compute 
with torsion coefficients (in which case they can often be compared to \'etale cohomology), 
and when the base field is algebraically closed. So here
we adopt the convention that by ``higher Chow groups with $\ell$-adic coefficients''
we mean the groups
\[\chow^*(X, n; \zz_\ell) \colonequals H_n\left(\varprojlim_m z^*(X_{\bar{k}}, \bullet) \otimes^L \zz/\ell^m \zz \right),\]
where $X_{\bar{k}}$ denotes the base change of $X$ to the algebraic closure $\bar{k}$ of our base field $k$.
Note that, when the $\chow^*(X_{\bar{k}}, n; \zz/\ell^m \zz)$ are finitely generated, we have
\[\chow^*(X, n; \zz_\ell) \simeq \varprojlim_m \chow^*(X_{\bar{k}}, n; \zz/\ell^m \zz).\]
In particular, since we have shown $\chow^*(\Delta_1)$ and $\chow^*(\mm)$ are finitely generated,
we still have
\[\chow^*(X, 0; \zz_\ell) = \chow^*(X_{\bar{k}}) \otimes \zz_\ell \tfor X \in \{\Delta_1, \mm, \mm \smallsetminus \Delta_1\}.\]
Moreover, Bloch's higher Chow groups satisfy $\chow^*(\spec \bar{k}, 1; \zz/\ell^m\zz) = 0$ by Theorem~10.3 of~\cite{pt}
(in combination with the isomorphism given in Lecture~19 of~\cite{pt} between higher Chow groups and motivic cohomology).
Consequently, this definition of
``higher Chow groups with $\ell$-adic coefficients''
satisfies $\chow^*(\spec k, 1; \zz_\ell) = 0$.

The description of $\mm \smallsetminus \Delta_1$ given
in Section~\ref{sec:u} shows that
\begin{align*}
(\mm \smallsetminus \Delta_1) / \mathbb{G}_m &\simeq (\pp V_6(2) \smallsetminus \{\text{forms with triple roots}\}) / \gl_2 \\
&\simeq (\pp V_6 \smallsetminus \{\text{forms with triple roots}\}) / \gl_2.
\end{align*}
\begin{defi}
Write $t = c_1(\oo_{\pp V_6}(1))$.
\end{defi}

So by construction,
$\mm \smallsetminus \Delta_1$ is a $\mathbb{G}_m$-bundle
over $(\mm \smallsetminus \Delta_1) / \mathbb{G}_m$
given by the complement of the zero section of a line bundle
with Chern class $t - 2\lambda_1$.

\begin{lm} \label{groth}
The stack $\pp V_6 / \gl_2$ has Chow ring
$\chow^*(\pp V_6 / \gl_2) = \zz[\lambda_1, \lambda_2, t] / p(\lambda_1, \lambda_2, t)$, where
\[p(\lambda_1, \lambda_2, t) = (t^2 - 6\lambda_1 t + 36\lambda_2) (t^2 - 6\lambda_1 t + 5\lambda_1^2 + 16 \lambda_2) (t^2 - 6\lambda_1 t + 8\lambda_1^2 + 4\lambda_2) (t - 3\lambda_1).\]
Moreover, its first higher Chow groups with $\ell$-adic coefficients vanish:
\[\chow^*(\pp V_6 / \gl_2, 1; \zz_\ell) = 0.\]
\end{lm}
\begin{proof}
Writing $\alpha$ and $\beta$ for the Chern roots of $V^*$
(with $\alpha + \beta = -\lambda_1$ and $\alpha \beta = \lambda_2$),
Grothendieck's formula
for the Chow ring of a projective bundle gives
\[\chow^* (\pp V_6/\gl_2) = \chow^* (B\! \gl_2) [t] / p(t) = \zz[\lambda_1, \lambda_2][t] / p(t),\]
where
\begin{align*}
p(t) &= (t + 6\alpha)(t + 6\beta) (t + 5\alpha + \beta)(t + \alpha + 5\beta) (t + 4\alpha + 2\beta)(t + 2\alpha + 4\beta) (t + 3\alpha + 3\beta) \\
&= (t^2 - 6\lambda_1 t + 36\lambda_2) (t^2 - 6\lambda_1 t + 5\lambda_1^2 + 16 \lambda_2) (t - 6\lambda_1 + 8\lambda_1^2 + 4\lambda_2) (t - 3\lambda_1).
\end{align*}

Moreover the vanishing of first higher Chow groups is preserved under taking
vector bundles, and thus (using the localization sequence)
Grassmannian bundles.
Since we have $\chow^*(\spec k, 1; \zz_\ell) = 0$,
we conclude that the first higher Chow groups with $\ell$-adic coefficients
vanish for $B\! \gl_2$, and thus for
any projective bundle over
$B\! \gl_2$, including $\pp V_6/\gl_2$.
\end{proof}

\begin{defi} Write $T_1 \subset \pp V_6/\gl_2$ for the locus with
exactly one triple root, and write $T_2 \subset \pp V_6/\gl_2$ for the locus
of perfect cubes.
\end{defi}

\begin{defi}
Write $\text{cub}_1 \colon (\pp V_1/\gl_2) \to (\pp V_3/\gl_2)$ and $\text{cub}_2 \colon (\pp V_2 /\gl_2) \to (\pp V_6/\gl_2)$
for the cubing maps.
Denote by $S_{ij} = s_1^i \times s_3^j$ (recall Definition~\ref{skr}), and by $s_{ij}$ the pushforward of $S_{ij}$ under the map
\[\theta \colon (\pp V_1 \times \pp V_3)/\gl_2 \to \pp V_6/\gl_2 \quad\text{defined by}\quad (f, g) \mapsto f^3 \cdot g.\]
\end{defi}

\noindent
Observe that we have the following commutative diagram, 
\[\begin{CD}
(\pp V_1 \times \pp V_1)/\gl_2 @>\mathbf{1} \times \text{cub}_1>> (\pp V_1 \times \pp V_3)/\gl_2 \\
@V{\text{multiplication}}VV @VV{\theta}V \\
\pp V_2 /\gl_2 @>\text{cub}_2>> \pp V_6 / \gl_2 \\
\end{CD}\]
In particular,
taking the images of the generators $s_1^i \times s_1^j$ we obtain
from Lemmas~\ref{lm:mult} and~\ref{ver}
(with $c_1 = -\lambda_1$ and $c_2 = \lambda_2$):
\begin{align*}
2 \text{cub}_2(s_2^0) &= 3 s_{02} - 6 \lambda_1 s_{01} + 6 (\lambda_1^2 - \lambda_2) s_{00} \\
\text{cub}_2(s_2^1) &= 3 s_{12} - 6 \lambda_1 s_{11} + 6 (\lambda_1^2 - \lambda_2) s_{10}\\
\text{cub}_2(s_2^1) &= s_{03} - 6 \lambda_2 s_{01} + 6 \lambda_1 \lambda_2 s_{00} \\
\text{cub}_2(s_2^2) &= s_{13} - 6 \lambda_2 s_{11} + 6 \lambda_1 \lambda_2 s_{10}
\end{align*}
Note that this forces the coefficients of $s_{02}$ to be even; we can thus write
\[s_{02}' = \frac{s_{02}}{2}.\]

\begin{lm} \label{aw} The Chow ring $\chow^*((\pp V_6 \smallsetminus T_2) / \gl_2)$ is generated as a module over $\zz[\lambda_1, \lambda_2]$
by the elements $s_6^0$, $s_6^1$, $s_6^2$, $s_6^3$, $s_6^4$, $s_6^5$, and $s_6^6$, subject to the relations
\begin{align*}
3 \cdot [s_{02}' - \lambda_1 s_{01} + (\lambda_1^2 - \lambda_2) s_{00}] &= 0 \\
3 \cdot [s_{12} - 2 \lambda_1 s_{11} + 2 (\lambda_1^2 - \lambda_2) s_{10}] &= 0 \\
s_{13} - 6 \lambda_2 s_{11} + 6 \lambda_1 \lambda_2 s_{10} &= 0.
\end{align*}
Moreover, assuming these three elements are linearly independent over $\zz[\lambda_1, \lambda_2]$,
we have $\chow^*((\pp V_6 \smallsetminus T_2) / \gl_2, 1; \zz_\ell) = 0$.
\end{lm}
\begin{proof}
This is immediate from Lemmas~\ref{groth} and~\ref{basismod} together with the localization exact sequence:
\begin{multline*}
\cdots \to \chow^* (\pp V_6 / \gl_2, 1; \zz_\ell) = 0 \to \chow^*((\pp V_6 \smallsetminus T_2) / \gl_2, 1; \zz_\ell) \\
\to \chow^{*-4}(T_2)_{\bar{k}} \otimes \zz_\ell \to \chow^* (\pp V_6 / \gl_2)_{\bar{k}} \otimes \zz_\ell \to \chow^*((\pp V_6 \smallsetminus T_2) / \gl_2)_{\bar{k}} \otimes \zz_\ell \to 0. \qedhere
\end{multline*}
\end{proof}

\begin{lm} \label{lm:z} The Chow ring $\chow^*((\mm \smallsetminus \Delta_1) / \mathbb{G}_m)$ is generated as a module over $\zz[\lambda_1, \lambda_2]$
by the elements $s_6^0$, $s_6^1$, $s_6^2$, $s_6^3$, $s_6^4$, $s_6^5$, and $s_6^6$, subject to the relations
\[s_{00} = s_{10} = s_{01} = s_{11} = s_{02}' = s_{12} = s_{13} = 0.\]
Moreover, assuming these seven elements are linearly independent over $\zz[\lambda_1, \lambda_2]$,
we have $\chow^*((\mm \smallsetminus \Delta_1) / \mathbb{G}_m, 1; \zz_\ell) = 0$.
\end{lm}
\begin{proof}
From the localization exact sequence
\[\cdots \to \chow^* ((\pp V_1 \times \pp V_1) / \gl_2) \to \chow^* ((\pp V_1 \times \pp V_3) / \gl_2) \to \chow^* (T_1) \to 0,\]
it follows that $\chow^*(T_1)$ is generated as a module over $\zz[\lambda_1, \lambda_2]$
by $S_{00}$, $S_{01}$, $S_{02}$, $S_{03}$, $S_{10}$, $S_{11}$, $S_{12}$, and $S_{13}$,
with the relations
\begin{align*}
3 S_{02} - 6 \lambda_1 S_{01} + 6 (\lambda_1^2 - \lambda_2) S_{00} &= 0 \\
3 S_{12} - 6 \lambda_1 S_{11} + 6 (\lambda_1^2 - \lambda_2) S_{10} &= 0\\
S_{03} - 6 \lambda_2 S_{01} + 6 \lambda_1 \lambda_2 S_{00} &= 0 \\
S_{13} - 6 \lambda_2 S_{11} + 6 \lambda_1 \lambda_2 S_{10} &= 0.
\end{align*}
Equivalently, eliminating $S_{03}$ and $S_{13}$ via the final two relations and writing
\[x = S_{02} - 2 \lambda_1 S_{01} + 2 (\lambda_1^2 - \lambda_2) S_{00} \tand y = S_{12} - 2 \lambda_1 S_{11} + 2 (\lambda_1^2 - \lambda_2) S_{10},\]
the Chow ring $\chow^*(T_1)$ is generated as a module over $\zz[\lambda_1, \lambda_2]$ by
$S_{00}$, $S_{01}$, $S_{10}$, $S_{11}$, $x$ and $y$, with relations
\[3x = 3y = 0.\]

Note that the assumed independence of $s_{00}$, $s_{10}$, $s_{01}$, $s_{11}$, $s_{02}'$, $s_{12}$, and $s_{13}$
implies the independence assumed in Lemma~\ref{aw}.
By Lemma~\ref{aw}, the Chow ring $\chow^*((\pp V_6 \smallsetminus T_2) / \gl_2)$ is generated as a module over $\zz[\lambda_1, \lambda_2]$ by
$s_6^0$, $s_6^1$, $s_6^2$, $s_6^3$, $s_6^4$, $s_6^5$, and $s_6^6$, subject to the relations
\begin{align*}
3x' &= 0 \\
3 \cdot \theta(y) &= 0 \\
s_{13} - 6 \lambda_2 s_{11} + 6 \lambda_1 \lambda_2 s_{10} &= 0,
\end{align*}
where $\theta(x) = 2x'$, and $\chow^*((\pp V_6 \smallsetminus T_2) / \gl_2, 1; \zz_\ell) = 0$.
Using the localization exact sequence
\[\chow^{*-2}(T_1) \to \chow^*((\pp V_6 \smallsetminus T_2) / \gl_2) \to \chow^*((\mm \smallsetminus \Delta_1) / \mathbb{G}_m) \to 0,\]
we conclude that
the Chow ring $\chow^* ((\mm \smallsetminus \Delta_1) / \mathbb{G}_m)$ is generated as a module over $\zz[\lambda_1, \lambda_2]$
by the elements $s_6^0$, $s_6^1$, $s_6^2$, $s_6^3$, $s_6^4$, $s_6^5$, and $s_6^6$, subject to the relations
\[s_{00} = s_{10} = s_{01} = s_{11} = x' = \theta(y) = s_{13} - 6 \lambda_2 s_{11} + 6 \lambda_1 \lambda_2 s_{10} = 0.\]
Moreover, assuming these seven elements are linearly independent over $\zz[\lambda_1, \lambda_2]$,
the localization sequence
\begin{multline*}
\cdots \to \chow^*((\pp V_6 \smallsetminus T_2) / \gl_2, 1; \zz_\ell) \to \chow^*((\mm \smallsetminus \Delta_1) / \mathbb{G}_m, 1; \zz_\ell) \\
\to \chow^{*-2}(T_1)_{\bar{k}} \otimes \zz_\ell \to \chow^*((\pp V_6 \smallsetminus T_2) / \gl_2)_{\bar{k}} \otimes \zz_\ell \to \cdots
\end{multline*}
implies $\chow^*((\mm \smallsetminus \Delta_1) / \mathbb{G}_m, 1; \zz_\ell) = 0$.
Since
\[x' = s_{02}' - \lambda_1 s_{01} + (\lambda_1^2 - \lambda_2) s_{00} \tand y = s_{12} - 2 \lambda_1 s_{11} + 2 (\lambda_1^2 - \lambda_2) s_{10},\]
the submodule generated by these seven elements is the same as the
submodule generated by the seven elements given in the statement of the lemma.
\end{proof}

We now combine Lemmas~\ref{lm:mult} and~\ref{ver}
to calculate $s_{ij}$ by pushing forward along $\theta$.
For this, we factor $\theta$ as cubing the first factor followed by multiplication:
\[(\pp V_1 \times \pp V_3)/\gl_2 \to (\pp V_3 \times \pp V_3)/\gl_2 \to \pp V_6/\gl_2.\]
The results are as follows:
\begin{align*}
s_{10} \colon s_1^1 \times s_3^0 &\mapsto [s_3^3 - 6 \lambda_2 s_3^1 + 6\lambda_1 \lambda_2 s_3^0] \times s_6^0 \mapsto s_6^3 - 60 \lambda_2 s_6^1 + 120\lambda_1\lambda_2 s_6^0 \\
s_{11} \colon s_1^1 \times s_3^1 &\mapsto [s_3^3 - 6 \lambda_2 s_3^1 + 6\lambda_1 \lambda_2 s_3^0] \times s_6^1 \mapsto s_6^4 - 36 \lambda_2 s_6^2 + 60\lambda_1\lambda_2 s_6^1 \\
s_{12} \colon s_1^1 \times s_3^2 &\mapsto [s_3^3 - 6 \lambda_2 s_3^1 + 6\lambda_1 \lambda_2 s_3^0] \times s_6^2 \mapsto s_6^5 - 18 \lambda_2 s_6^3 + 24\lambda_1\lambda_2 s_6^2 \\
s_{13} \colon s_1^1 \times s_3^3 &\mapsto [s_3^3 - 6 \lambda_2 s_3^1 + 6\lambda_1 \lambda_2 s_3^0] \times s_6^3 \mapsto s_6^6 - 6 \lambda_2 s_6^4 + 6\lambda_1\lambda_2 s_6^3 \\
s_{00} \colon s_1^0 \times s_3^0 &\mapsto [3 s_3^2 - 6 \lambda_1 s_3^1 + 6(\lambda_1^2 - \lambda_2)s_3^0] \times s_3^0 \mapsto 12 s_6^2 - 60 \lambda_1 s_6^1 + 120(\lambda_1^2 - \lambda_2)s_6^0 \\
s_{01} \colon s_1^0 \times s_3^1 &\mapsto [3 s_3^2 - 6 \lambda_1 s_3^1 + 6(\lambda_1^2 - \lambda_2)s_3^0] \times s_3^1 \mapsto 9 s_6^3 - 36 \lambda_1 s_6^2 + 60(\lambda_1^2 - \lambda_2)s_6^1 \\
&\Rightarrow s_{02}' = 3 s_6^4 - 9 \lambda_1 s_6^3 + 12(\lambda_1^2 - \lambda_2)s_6^2.
\end{align*}

\begin{lm} We have $\chow^*((\mm \smallsetminus \Delta_1) / \mathbb{G}_m, 1; \zz_\ell) = 0$.
\end{lm}
\begin{proof}
We just have to check that $s_{10}$, $s_{11}$, $s_{12}$, $s_{13}$, $s_{00}$, $s_{01}$, and $s_{02}'$
are linearly independent over $\zz[\lambda_1, \lambda_2]$.
For this, we use the above expressions for them in the basis
$s_6^0$, $s_6^1$, $s_6^2$, $s_6^3$, $s_6^4$, $s_6^5$, and $s_6^6$,
which reduces our problem to checking that the following matrix is nonsingular:
\[\begin{array}{ccccccc}
120 \lambda_1 \lambda_2 & -60 \lambda_2 & 0 & 1 & 0 & 0 & 0 \\
0 & 60 \lambda_1 \lambda_2 & -36 \lambda_2 & 0 & 1 & 0 & 0\\
0 & 0 & 24 \lambda_1 \lambda_2 & -18 \lambda_2 & 0 & 1 & 0 \\
0 & 0 & 0 & 6 \lambda_1 \lambda_2 & -6 \lambda_2 & 0 & 1 \\
120(\lambda_1^2 - \lambda_2) & -60 \lambda_1 & 12 & 0 & 0 & 0 & 0 \\
0 & 60(\lambda_1^2 - \lambda_2) & -36 \lambda_1 & 9 & 0 & 0 & 0 \\
0 & 0 & 12(\lambda_1^2 - \lambda_2) & -9 \lambda_1 & 3 & 0 & 0 \\
\end{array}\]
It thus remains to note that this matrix has determinant
$86400(\lambda_1^2 - 4\lambda_2)^3 \neq 0$.
\end{proof}

\begin{lm} The Chow ring of $(\mm \smallsetminus \Delta_1) / \mathbb{G}_m$ is given by
\[\chow^*((\mm \smallsetminus \Delta_1) / \mathbb{G}_m) = \zz[\lambda_1, \lambda_2, t] / (s_{00}, s_{10}, s_{02}').\]
\end{lm}
\begin{proof}
As $\theta^* t$ restricts to a hyperplane class on the $\pp^3$ factor,
push-pull implies $s_{ij} \in (s_{00}, s_{10})$ for any $i$ and $j$,
as well as $p \in (s_{00}, s_{10})$. The desired result now follows from Lemmas~\ref{lm:z}
and~\ref{groth}.
\end{proof}

To compute the Chow ring of $(\mm \smallsetminus \Delta_1) / \mathbb{G}_m$ more explicitly, we first use the recursion
of Lemma~\ref{rec}.
to compute first few the~$s_6^j$:
\begin{align*}
s_6^0 &= 1 \\
s_6^1 &= t \\
s_6^2 &= t^2 - \lambda_1 t + 6 \lambda_2 \\
s_6^3 &= t^3 - 3 \lambda_1 t^2 + (2 \lambda_1^2 + 16 \lambda_2) t - 12 \lambda_1 \lambda_2 \\
s_6^4 &= t^4 - 6\lambda_1 t^3 + (11 \lambda_1^2 + 28 \lambda_2) t^2 + (-6\lambda_1^3 - 72\lambda_1 \lambda_2) t + 36 \lambda_1^2 \lambda_2 + 72 \lambda_2^2.
\end{align*}
Substituting these into the expressions for $s_{ij}$ given above,
we obtain:
\begin{align*}
s_{10} &= t^3 - 3 \lambda_1 t^2 + (2 \lambda_1^2 - 44 \lambda_2) t + 108 \lambda_1 \lambda_2 \\
&= 20 \lambda_1 \lambda_2 + (t^2 - \lambda_1 t - 44 \lambda_2) \cdot (t - 2 \lambda_1) \\
s_{00} &= 12 t^2 - 72 \lambda_1 t + 120\lambda_1^2 - 48 \lambda_2 \\
&= 24 \lambda_1^2 - 48 \lambda_2 + (12 t - 48 \lambda_1) \cdot (t - 2 \lambda_1)\\
s_{02}' &= 3 t^4 - 27\lambda_1 t^3 + (72 \lambda_1^2 + 72 \lambda_2) t^2 - (48 \lambda_1^3 + 348 \lambda_1 \lambda_2) t + 288 \lambda_1^2 \lambda_2 + 144 \lambda_2^2 \\
&= -60(\lambda_1^2 - 4\lambda_2)(t - 3 \lambda_1) \cdot (t - 2 \lambda_1) + (3t - 6 \lambda_1) s_{10} - (\lambda_1 t - 3 \lambda_1^2 + 3 \lambda_2) s_{00}.
\end{align*}
In other words, the Chow ring of $(\mm \smallsetminus \Delta_1) / \mathbb{G}_m$ is generated by $\lambda_1$, $\lambda_2$, and $t$ with relations:
\begin{align*}
20 \lambda_1 \lambda_2 + (t^2 - \lambda_1 t - 44 \lambda_2) \cdot (t - 2 \lambda_1) &= 0 \\
24 \lambda_1^2 - 48 \lambda_2 + (12 t - 48 \lambda_1) \cdot (t - 2 \lambda_1) &= 0 \\
60(\lambda_1^2 - 4\lambda_2)(t - 3 \lambda_1) \cdot (t - 2 \lambda_1) &= 0.
\end{align*}

\begin{thm} \label{45} The Chow ring of $\mm \smallsetminus \Delta_1$ is given by:
\[\chow^*(\mm \smallsetminus \Delta_1) = \zz[\lambda_1, \lambda_2] / (24 \lambda_1^2 - 48 \lambda_2, 20 \lambda_1 \lambda_2).\]
In addition, $\chow^*(\mm \smallsetminus \Delta_1, 1; \zz_\ell)$ 
is generated by two classes in degrees $4$ and $5$ respectively.
\end{thm}
\begin{proof}
Note that $\mm \smallsetminus \Delta_1$ is a $\mathbb{G}_m$ bundle over $(\mm \smallsetminus \Delta_1) / \mathbb{G}_m$, with Chern class
$t - 2 \lambda_1$. From the localization sequence
\[\chow^*((\mm \smallsetminus \Delta_1) / \mathbb{G}_m) \to \chow^*((\mm \smallsetminus \Delta_1) / \mathbb{G}_m) \to \chow^*(\mm \smallsetminus \Delta_1) \to 0,\]
we learn that $\chow^*(\mm \smallsetminus \Delta_1)$ is given by the cokernel
of multiplication by $t - 2\lambda_1$ on the Chow ring $\chow^*((\mm \smallsetminus \Delta_1) / \mathbb{G}_m)$,
which is evidently the ring given in the statement of the theorem.

Moreover, from the localization sequence
\begin{multline*}
\cdots \to \chow^*((\mm \smallsetminus \Delta_1) / \mathbb{G}_m, 1; \zz_\ell) = 0 \to \chow^*(\mm \smallsetminus \Delta_1, 1; \zz_\ell) \\
\to \chow^{*-1}((\mm \smallsetminus \Delta_1) / \mathbb{G}_m)_{\bar{k}} \otimes \zz_\ell \to \chow^*((\mm \smallsetminus \Delta_1) / \mathbb{G}_m)_{\bar{k}} \otimes \zz_\ell \to \cdots,
\end{multline*}
we learn that $\chow^*(\mm \smallsetminus \Delta_1, 1; \zz_\ell)$
is given by the kernel of multiplication by
$t - 2\lambda_1$ on $\chow^{*-1}((\mm \smallsetminus \Delta_1) / \mathbb{G}_m)_{\bar{k}} \otimes \zz_\ell$.
This kernel is
evidently generated by the class
\[60(\lambda_1^2 - 4\lambda_2)(t - 3 \lambda_1)\]
in degree $3$, and the class
\[5 \lambda_1 \lambda_2 \cdot (12 t - 48 \lambda_1) - (6 \lambda_1^2 - 12 \lambda_2) \cdot (t^2 - \lambda_1 t - 44 \lambda_2)\]
in degree $4$.
\end{proof}

Observe that the twisting
$\mathbb{G}_m$-action extends over the whole of $\mm$.
Indeed, consider the universal family of curves over $\mm$, pulled back to $\mm \times \mathbb{G}_m$.
There is a natural action of $\zz/2\zz$ via the hyperelliptic involution on the universal family over $\mm$
composed with
multiplication by $-1$ on $\mathbb{G}_m$;
the quotient of this family by this action gives a family of curves over
$\mm \times (\mathbb{G}_m / (\zz/2\zz))$ --- or, using the isomorphism $(\mathbb{G}_m / (\zz/2\zz)) \simeq \mm \times \mathbb{G}_m$
given by the squaring map --- a family of curves over $\mm \times \mathbb{G}_m$, i.e.\
an action of $\mathbb{G}_m$ on $\mm$.
By inspection, this restricts to the twisting action defined above
on $\mm \smallsetminus \Delta_1$.

Our next task is to leverage
the fact that the $\mathbb{G}_m$-action extends
over the whole of $\mm$ to compute
the image of the degree $5$ generator in $\chow^5(\mm \smallsetminus \Delta_1, 1; \zz_\ell)$
in $\chow^4(\Delta_1)_{\bar{k}} \otimes \zz_\ell$ under the boundary map in the localization sequence
\[\cdots \to \chow^5(\mm \smallsetminus \Delta_1, 1; \zz_\ell) \to \chow^4(\Delta_1)_{\bar{k}} \otimes \zz_\ell \to \chow^5(\mm)_{\bar{k}} \otimes \zz_\ell \to \cdots.\]
First observe that,
since $\chow^*(\mm \smallsetminus \Delta_1, 1; \zz_\ell)$ is generated in degrees
$4$ and $5$, any degree~$2$ or~$3$ class
$x \in \chow^*(\mm)_{\bar{k}} \otimes \zz_\ell$ whose image in $\chow^*(\mm \smallsetminus \Delta_1)_{\bar{k}} \otimes \zz_\ell$ vanishes
is the image of a unique degree~$1$ or~$2$ class in $\chow^*(\Delta_1)_{\bar{k}} \otimes \zz_\ell$,
which we denote $\bar{x}$.

\begin{lm} \label{im5}
The image of the degree $5$ generator in $\chow^5(\mm \smallsetminus \Delta_1, 1; \zz_\ell)$
in $\chow^4(\Delta_1)_{\bar{k}} \otimes \zz_\ell$ 
is given by
\[5 \lambda_1 \lambda_2 \cdot (\overline{24 \lambda_1^2 - 48 \lambda_2}) - (6 \lambda_1^2 - 12 \lambda_2) \cdot \overline{20 \lambda_1 \lambda_2}.\]
\end{lm}
\begin{proof}
Consider the digram \\
\scalebox{0.89}{\parbox{\textwidth}{
$$\begin{CD}
0 @>>> \chow^* (\Delta_1 / \mathbb{G}_m)_{\bar{k}} \otimes \zz_\ell @>>> \chow^* (\mm/\mathbb{G}_m)_{\bar{k}} \otimes \zz_\ell @>>> \chow^*((\mm \smallsetminus \Delta_1) / \mathbb{G}_m)_{\bar{k}} \otimes \zz_\ell @>>> 0 \\
@. @VVV @VVV @VVV @. \\
0 @>>> \chow^* (\Delta_1 / \mathbb{G}_m)_{\bar{k}} \otimes \zz_\ell @>>> \chow^* (\mm/\mathbb{G}_m)_{\bar{k}} \otimes \zz_\ell @>>> \chow^*((\mm \smallsetminus \Delta_1) / \mathbb{G}_m)_{\bar{k}} \otimes \zz_\ell @>>> 0,
\end{CD}$$}} \\
where the vertical maps are multiplication by $t - 2\lambda_1$; since $\chow^*((\mm \smallsetminus \Delta_1) / \mathbb{G}_m, 1; \zz_\ell) = 0$, the rows are exact.
The snake lemma then gives us a map
from the kernel of multiplication by $t - 2 \lambda_1$ on $\chow^*((\mm \smallsetminus \Delta_1) / \mathbb{G}_m)_{\bar{k}} \otimes \zz_\ell$ to the cokernel of multiplication
by $t - 2 \lambda_1$ on $\chow^* (\Delta_1 / \mathbb{G}_m)_{\bar{k}} \otimes \zz_\ell$;
under the identification of the kernel of multiplication by $t - 2 \lambda_1$ on $\chow^*((\mm \smallsetminus \Delta_1) / \mathbb{G}_m)_{\bar{k}} \otimes \zz_\ell$
with $\chow^*(\mm \smallsetminus \Delta_1, 1; \zz_\ell) = 0$,
and the identification of the cokernel of multiplication by $t - 2 \lambda_1$ on $\chow^* (\Delta_1 / \mathbb{G}_m)_{\bar{k}} \otimes \zz_\ell$
with $\chow^* (\Delta_1)_{\bar{k}} \otimes \zz_\ell$, this map
corresponds to the boundary map appearing the in the localization exact sequence.

We have a natural lift of $t - 2\lambda_1$ (and thus $t$) corresponding to the extension of the $\mathbb{G}_m$-action.
Given a class in $x \in \chow^* (\mm / \mathbb{G}_m)_{\bar{k}} \otimes \zz_\ell$, whose image in $\chow^*((\mm \smallsetminus \Delta_1) / \mathbb{G}_m)_{\bar{k}} \otimes \zz_\ell$ is zero,
we also write $\bar{x}$ for the class in $\chow^* (\Delta_1 / \mathbb{G}_m)_{\bar{k}} \otimes \zz_\ell$
mapping to $x$.

Now the map arising from the snake lemma can be described as follows: Given a class in the kernel of
multiplication by $t - 2 \lambda_1$ on $\chow^*((\mm \smallsetminus \Delta_1) / \mathbb{G}_m)_{\bar{k}} \otimes \zz_\ell$, we pick a lift of this class to $\chow^* (\mm / \mathbb{G}_m)_{\bar{k}} \otimes \zz_\ell$;
we multiply that lift by $t - 2 \lambda_1$; we write that as the image (uniquely)
of a class in $\chow^* (\Delta_1 / \mathbb{G}_m)_{\bar{k}} \otimes \zz_\ell$; and finally we reduce that class
modulo $t - 2\lambda_1$ to obtain a class in the cokernel
of multiplication by $t - 2 \lambda_1$ on $\chow^* (\Delta_1 / \mathbb{G}_m)_{\bar{k}} \otimes \zz_\ell$.
Following that recipe for our degree $5$ class, we obtain
\[\overline{(t - 2 \lambda_1) \cdot [5 \lambda_1 \lambda_2 \cdot (12 t - 48 \lambda_1) - (6 \lambda_1^2 - 12 \lambda_2) \cdot (t^2 - \lambda_1 t - 44 \lambda_2)]} \mod t - 2\lambda_1.\]
Writing
\begin{multline*}
(t - 2 \lambda_1) \cdot [5 \lambda_1 \lambda_2 \cdot (12 t - 48 \lambda_1) - (6 \lambda_1^2 - 12 \lambda_2) \cdot (t^2 - \lambda_1 t - 44 \lambda_2)] \\
= (5 \lambda_1 \lambda_2) \cdot [24 \lambda_1^2 - 48 \lambda_2 + (12 t - 48 \lambda_1)(t - 2 \lambda_1)] \\
- (6 \lambda_1^2 - 12 \lambda_2) \cdot [20 \lambda_1 \lambda_2 + (t^2 - \lambda_1 t - 44 \lambda_2)(t - 2 \lambda_1)],
\end{multline*}
we obtain
\begin{multline*}
(5 \lambda_1 \lambda_2) \cdot [\overline{24 \lambda_1^2 - 48 \lambda_2 + (12 t - 48 \lambda_1)(t - 2 \lambda_1)} \mod t - 2\lambda_1] \\
- (6 \lambda_1^2 - 12 \lambda_2) \cdot [\overline{20 \lambda_1 \lambda_2 + (t^2 - \lambda_1 t - 44 \lambda_2)(t - 2 \lambda_1)} \mod t - 2\lambda_1].
\end{multline*}

Since $\overline{24 \lambda_1^2 - 48 \lambda_2 + (12 t - 48 \lambda_1)(t - 2 \lambda_1)}$ mod $t - 2\lambda_1$
is a class in degree $2$, and the map $\chow^2(\Delta_1)_{\bar{k}} \otimes \zz_\ell \to \chow^2(\mm)_{\bar{k}} \otimes \zz_\ell$ is injective,
this class is determined by its image in $\chow^2(\mm)_{\bar{k}} \otimes \zz_\ell$ --- which is just $24 \lambda_1^2 - 48 \lambda_2$.
Consequently this class can be written simply as $\overline{24 \lambda_1^2 - 48 \lambda_2}$.
Reasoning similarly for the second term, we deduce the formula given in the statement of the theorem.
\end{proof}

\section{\label{sec:tofinite} Reduction to three nonvanishing statements}

In this section, we show Theorem~\ref{thm:main} holds, provided that the classes
\[\delta_1(\delta_1 + \lambda_1) \lambda_1^2, \quad \delta_1(\delta_1 + \lambda_1) \lambda_2, \tand \delta_1(\delta_1 + \lambda_1) (\lambda_1^2 + \lambda_2)\]
are all nonvanishing in $\chow^*(\mm) \otimes \zz/2\zz$
(which we assume for the remainder of this section).
Since
\[\chow^*(\Delta_1) = \zz[\lambda_1, \lambda_2, \gamma] / (2 \gamma, \gamma^2 + \lambda_1 \gamma, 24 \lambda_1^2 - 48 \lambda_2, 24 \lambda_1 \lambda_2) \]
is invariant under extension of the base field from $k$ to $\bar{k}$,
and only has torsion coprime to the characteristic (as we have assumed
the characteristic is distinct from $2$ and $3$),
the injectivity of
\[\chow^*(\Delta_1)_{\bar{k}} \otimes \zz_\ell \to \chow^{*+1}(\mm)_{\bar{k}} \otimes \zz_\ell\]
for $\ell$ coprime to the characteristic implies the injectivity of
\[\chow^*(\Delta_1) \to \chow^{*+1}(\mm).\]

\subsection{\boldmath Injectivity in degrees $0$, $1$, and $2$}
By Theorem~\ref{45},
\[\chow^*(\Delta_1)_{\bar{k}} \otimes \zz_\ell \to \chow^{*+1}(\mm)_{\bar{k}} \otimes \zz_\ell\]
is injective in degrees $0$, $1$, and $2$, for all $\ell$ coprime to the characteristic.
Consequently, $\chow^*(\Delta_1) \to \chow^{*+1}(\mm)$ is injective in degrees $0$, $1$, and~$2$.

\subsection{Relations from Grothendieck--Riemann--Roch}
Here we show that the relations established in Section~\ref{relgrr}
in characteristic zero, and modulo torsion in arbitrary characteristic,
in fact hold exactly in arbitrary characteristic distinct from $2$ and $3$.

The injectivity of
$\chow^0(\Delta_1) \to \chow^1(\mm)$,
combined with Theorems~\ref{adelta1} and~\ref{45}, implies that
$\chow^1(\mm)$ is torsion-free.
Thus, the formula in Lemma~\ref{delta0} holds exactly, not merely modulo torsion.
Combined with Lemma~\ref{delta0lambda2}, this implies the relation
\begin{equation} \label{rel3}
20 \lambda_1 \lambda_2 - 4 \delta_1 \lambda_2 = 0.
\end{equation}

Next, the relation $24\lambda_1^2 - 48 \lambda_2 = 0$
holds on $\mm \smallsetminus \Delta_1$ by Theorem~\ref{45},
and holds on $\mm$ modulo torsion by Lemma~\ref{lm2448}.
Since $\chow^1(\Delta_1) \to \chow^2(\mm)$ is injective,
$24\lambda_1^2 - 48 \lambda_2$ must equal the pushforward
of a torsion class in $\chow^1(\Delta_1)$.
But the only nonzero torsion class in $\chow^1(\Delta_1)$ is $\gamma$,
whose pushforward is $\delta_1(\delta_1 + \lambda_1)$ by Lemma~\ref{delta1delta1}.
Thus either
\[24\lambda_1^2 - 48 \lambda_2 = 0 \tor 24\lambda_1^2 - 48 \lambda_2 = \delta_1(\delta_1 + \lambda_1).\]
This second relation implies $\delta_1(\delta_1 + \lambda_1) = 0 \in \chow^2(\mm) \otimes \zz/2\zz$,
which contradicts our assumption that
$\delta_1(\delta_1 + \lambda_1) \lambda_1^2 \neq 0 \in \chow^4(\mm) \otimes \zz/2\zz$.
Consequently the first relation must hold, i.e.\
\begin{equation} \label{rel2}
24\lambda_1^2 - 48 \lambda_2 = 0.
\end{equation}

\subsection{\boldmath Injectivity in degree $3$}
By Lemma~\ref{delta1delta1}, the kernel of $\chow^3(\Delta_1) \to \chow^4(\mm)$
is contained in the kernel of multiplication by $\gamma - \lambda_1$.
But by inspection, there are only three nonzero elements of this later kernel:
$\gamma \lambda_1^2$, and $\gamma \lambda_2$, and $\gamma(\lambda_1^2 + \lambda_2)$.
Pushing these classes forward to $\chow^4(\mm)$ using Lemma~\ref{delta1delta1},
we obtain exactly the three classes whose reduction modulo $2$
we have assumed is nonvanishing.

\subsection{Injectivity in higher degrees}
Combining the injectivity we have established in degrees $0$, $1$, $2$, and $3$
with Theorem~\ref{45} and Lemma~\ref{im5}, it remains to show
\[5 \lambda_1 \lambda_2 \cdot (\overline{24 \lambda_1^2 - 48 \lambda_2}) - (6 \lambda_1^2 - 12 \lambda_2) \cdot \overline{20 \lambda_1 \lambda_2} = 0 \in \chow^4(\Delta_1).\]
But using the relations \eqref{rel2} and \eqref{rel3}, we have
\begin{align*}
5 \lambda_1 \lambda_2 \cdot (\overline{24 \lambda_1^2 - 48 \lambda_2}) - (6 \lambda_1^2 - 12 \lambda_2) \cdot \overline{20 \lambda_1 \lambda_2} &= 5 \lambda_1 \lambda_2 \cdot 0 - (6 \lambda_1^2 - 12 \lambda_2) \cdot 4 \lambda_2 \\
&= -\lambda_2 \cdot (24 \lambda_1^2 - 48 \lambda_2) \\
&= 0.
\end{align*}

\subsection{Conclusion of proof}
The injectivity established thus far implies we have a short exact sequence
\[0 \to \chow^{*-1}(\Delta_1) \to \chow^*(\mm) \to \chow^*(\mm \smallsetminus \Delta) \to 0,\]
and we have previously (Theorems~\ref{adelta1} and~\ref{45}) established that
\begin{align*}
\chow^*(\Delta_1) &= \zz[\lambda_1, \lambda_2, \gamma] / (2 \gamma, \gamma^2 + \lambda_1 \gamma, 24 \lambda_1^2 - 48 \lambda_2, 24 \lambda_1 \lambda_2) \\
\chow^*(\mm \smallsetminus \Delta) &= \zz[\lambda_1, \lambda_2] / (24 \lambda_1^2 - 48 \lambda_2, 20 \lambda_1 \lambda_2).
\end{align*}
Given Lemma~\ref{delta1delta1}, and the relations \eqref{rel2} and \ref{rel3} that we have established which
extend the relations $24 \lambda_1^2 - 48 \lambda_2 = 0$ and $20 \lambda_1 \lambda_2 = 0$ from $\mm \smallsetminus \Delta_1$
to all of $\mm$,
it follows that $\chow^*(\mm)$ is generated by $\lambda_1$, $\lambda_2$, and $\delta_1$
with relations given by \eqref{rel2}, \eqref{rel3}, and the pushforwards of the relations that hold
in $\chow^*(\Delta_1)$, i.e.:
\begin{align*}
24 \lambda_1^2 - 48 \lambda_2 &= 0 \\
20 \lambda_1 \lambda_2 - 4 \delta_1 \lambda_2 &= 0 \\
[2(\delta_1 + \lambda_1)] \cdot \delta_1 &= 0 \\
[(\delta_1 + \lambda_1)^2 + \lambda_1 (\delta_1 + \lambda_1)] \cdot \delta_1 &= 0 \\
[24 \lambda_1^2 - 48 \lambda_2] \cdot \delta_1 &= 0 \\
[24 \lambda_1 \lambda_2] \cdot \delta_1 &= 0.
\end{align*}
By inspection,
these relations generated exactly the ideal of relations claimed in Theorem~\ref{thm:main}.

\section{\label{sec:biell} Chow ring of the bielliptic stack}

As shown in the previous section,
in order to complete the proof of Theorem~\ref{thm:main}, all that remains is to show that the classes
\[\delta_1(\delta_1 + \lambda_1) \lambda_1^2, \quad \delta_1(\delta_1 + \lambda_1) \lambda_2, \tand \delta_1(\delta_1 + \lambda_1) (\lambda_1^2 + \lambda_2)\]
are all nonvanishing in $\chow^*(\mm) \otimes \zz/2\zz$.
We will do this by pulling these classes back to a certain test family,
given explicitly in Section~\ref{subsec:test}.

\subsection{Motivation for our test family}
(This section is, strictly speaking, not logically necessary.)

Recall that these classes are pushforwards of classes on $\Delta_1$,
which are multiples of the pullback to $\Delta_1$ of the generator of $\chow^*(B(\zz/2\zz))$,
where $\zz/2\zz$ acts by exchanging an ordered pair of elliptic curves.
The fixed point locus of this $\zz/2\zz$-action corresponds to curves in $\Delta_1$
consisting of two isomorphic elliptic curves joined at a point,
and on this locus, the $\zz/2\zz$-action gives rise to the automorphism
that exchanges the two components, which is a bielliptic involution.
This consideration suggests that the universal family of bielliptic curves
would serve as a good test family. To be precise:

\begin{defi} A \emph{bielliptic curve of genus $2$}
is a stable curve $C$ of genus $2$, together with an unordered pair $\{\iota, \iota'\}$
of bielliptic involutions which differ by multiplication by the hyperelliptic involution.
\end{defi}

\begin{lm}
Let $C$ be a bielliptic curve of genus $2$. Then $\dim H^0(\omega_C^2) = 3$,
and the hyperelliptic involution acts trivially on $H^0(\omega_C^2)$.
The bielliptic involution (which is well-defined modulo the hyperelliptic
involution) has a $2$-dimensional trivial eigenspace
and a $1$-dimensional nontrivial eigenspace.
\end{lm}
\begin{proof}
Since $\omega_C^{-1}$ has no sections when $C$ is a stable curve,
Riemann--Roch implies that $\dim H^0(\omega_C^2) = 3$ as desired.
The remaining conditions are then visibly both open and closed, so it suffices
to check them when $C$ is smooth.

In this case, any section of $\omega_C^2$
sent to its negative under the hyperelliptic involution must vanish
at the $6$ Weierstrass points; since $\omega_C^2$ is of degree $4$,
there are no such sections, and so the hyperelliptic involution
acts trivially.

For the action of the bielliptic involution,
write $f \colon C \to E$ for the quotient map.
Then there is an evident filtration $H^0(f^* \omega_E^2) \subset H^0(f^* \omega_E \otimes \omega_C) \subset H^0(\omega_C^2)$,
whose quotients are $1$-dimensional, on which the bielliptic involution
acts as $+1$, $-1$, and $+1$ respectively.
\end{proof}

Consequently, every bielliptic curve
comes equipped with a canonical map to $\pp^1$
defined by the sections of $H^0(\omega_C^2)$ invariant under the bielliptic
involution.
By the Riemann--Hurwitz formula, this map is branched over $5$ points
with multiplicity.

Moreover, taking the ramification points of the
quotients of $C$ by the two bielliptic involutions
mapping to $\pp^1$ defines two subsets of $4$ of these $5$ points.
Their intersection then gives a canonically-defined subset of $3$
points.
We conclude that the family of curves in weighted projective space given by
\begin{align*}
z^2 &= (ax + by) \cdot (ex^3 + fx^2y + gxy^2 + hy^3) \\
w^2 &= (cx + dy) \cdot (ex^3 + fx^2y + gxy^2 + hy^3),
\end{align*}
over the locus of triples of two degree $1$ polynomials
and one degree $3$ polynomial whose product has no triple roots,
contains every isomorphism class of bielliptic curves.

Moreover, since the map to $\pp^1$ is canonical, we conclude
that any isomorphism between two curves in this family
is via a linear change of variables on $x$ and $y$,
scalar multiplication and exchanging on $z$ and $w$,
and rescaling the cubic polynomial.
In other words, writing $G = (\cc^\times \times \cc^\times) \rtimes \zz/2\zz$,
any such isomorphism is via the action of $\gl_2 \times G \times \cc^\times$
with $\gl_2$ acting on $(x, y)$, and $G$ acting on $(z, w)$,
and $\cc^\times$ acting on $a, b, c, d$ with weight $-1$ and
$e, f, g, h$ with weight $+1$.
By inspection, the subgroup of $\gl_2 \times G \times \cc^\times$
acting trivially is the image of $\cc^\times$ under the map
$t \mapsto (t \cdot \mathbf{1}, t^4 \cdot \mathbf{1}, t^3)$.

We conclude that the stack of bielliptic curves
is isomorphic to the quotient
of the locus of triples of polynomials of degrees $1$, $1$, and $3$,
whose product has
no triple root,
by the natural action of $(\gl_2 \times G \times \cc^\times) / \cc^\times$ constructed above.

To recast this in a somewhat nicer form, we observe that
$(\gl_2 \times G \times \cc^\times) / \cc^\times$ is itself isomorphic to $\gl_2 \times G$
via the map $\gl_2 \times G \times \cc^\times \to \gl_2 \times G$ defined by
\[A \times B \times t \mapsto \frac{\det A}{t} \cdot A \times \left(\frac{\det A}{t}\right)^4 \cdot B\]
whose kernel is, by inspection, the image of $\cc^\times$ under the map
$t \mapsto (t \cdot \mathbf{1}, t^4 \cdot \mathbf{1}, t^3)$.
An inverse to this map is given by
\[A \times B \mapsto A \times B \times \det A.\]

\subsection{Our test family\label{subsec:test}}
Our test family will be the space
\[[V_3(1) \times (V_1(-1) \boxtimes W_{-2})] / (G \times \gl_2);\]
viewing elements of $V_3(1)$ as cubic polynomials $ex^3 + fx^2y + gxy^2 + hy^3$,
and elements of $V_1(-1) \boxtimes W_{-2}$ as pairs of linear polynomials $(ax + by, cx + dy)$,
there is a family of genus $2$ curves defined over this quotient by:
\begin{align*}
z^2 &= (ax + by) \cdot (ex^3 + fx^2y + gxy^2 + hy^3) \\
w^2 &= (cx + dy) \cdot (ex^3 + fx^2y + gxy^2 + hy^3),
\end{align*}
where $\gl_2$ acts naturally on $(x, y)$ and $G$ acts naturally on $(z, w)$.
This will serve as our test family; the discussion of the previous
subsection shows this is the universal family of bielliptic curves.

\subsection{The zero sections\label{subsec:zero}}
As in Section~\ref{sec:adelta1},
we will use the formulas in Section~\ref{sec:mult} to calculate
loci in the total space of vector bundles with the origin excised,
by pulling back
the corresponding classes from the projectivization.

So we begin by excising the origins of $V_3(1)$ and $V_1(-1) \boxtimes W_{-2}$ respectively;
as in Section~\ref{sec:mult}, this imposes relations given by the vanishing
of the Euler classes $e(V_3(1)) = c_4(V_3(1))$ and
$e(V_1(-1) \boxtimes W_{-2}) = c_4(V_1(-1) \boxtimes W_{-2})$.

Write $a_1$ and $a_2$ for the Chern
roots of $V$;
and $b_1$ and $b_2$ for the Chern roots of $W_2$.
Then the Euler class of $V_3(1)$ is
\begin{align*}
e(V_3(1)) &= (\alpha_1 - 3a_1)(\alpha_1 - 3a_2)(\alpha_1 - 2a_1 - a_2)(\alpha_1 - a_1 - 2a_2) \\
&= (\alpha_1^2 - 3(a_1 + a_2) \alpha_1 + 9 a_1 a_2)(\alpha_1^2 - 3(a_1 + a_2) \alpha_1 + 2(a_1 + a_2)^2 + a_1 a_2) \\
&= (\alpha_1^2 - 3\alpha_1^2 + 9\alpha_2)(\alpha_1^2 - 3\alpha_1^2 + 2\alpha_1^2 + \alpha_2) \\
&= 9 \alpha_2^2 - 2 \alpha_1^2 \alpha_2.
\end{align*}
Similarly, the Euler class of $V_1(-1) \boxtimes W_{-2}$ is
\begin{align*}
e(V_1(-1) \boxtimes W_{-2}) &= (-\alpha_1 - a_1 - b_1)(-\alpha_1 - a_2 - b_1)(-\alpha_1 - a_1 - b_2)(-\alpha_1 - a_2 - b_2) \\
&= \big((\alpha_1 + b_1)^2 + (\alpha_1 + b_1)(a_1 + a_2) + a_1a_2\big) \\
&\qquad \cdot \big((\alpha_1 + b_2)^2 + (\alpha_1 + b_2)(a_1 + a_2) + a_1a_2\big) \\
&= \big((\alpha_1 + b_1)^2 + (\alpha_1 + b_1)\alpha_1 + \alpha_2\big) \cdot \big((\alpha_1 + b_2)^2 + (\alpha_1 + b_2)\alpha_1 + \alpha_2\big) \\
&= (b_1 b_2)^2 + 3\alpha_1 (b_1 + b_2) b_1 b_2 + (2\alpha_1^2 + \alpha_2) (b_1 + b_2)^2 \\
&\qquad + (5\alpha_1^2 - 2\alpha_2) b_1 b_2 + (6 \alpha_1^3 + 3\alpha_1\alpha_2) (b_1 + b_2) + (2 \alpha_1^2 + \alpha_2)^2 \\
&= (4\beta_2)^2 + 3\alpha_1 (2\beta_1 + \gamma) (4 \beta_2) + (2\alpha_1^2 + \alpha_2)(2\beta_1 + \gamma)^2 \\
&\qquad + (5\alpha_1^2 - 2\alpha_2) (4\beta_2) + (6 \alpha_1^3 + 3\alpha_1\alpha_2) (2\beta_1 + \gamma) + (2 \alpha_1^2 + \alpha_2)^2 \\
&= 4\alpha_1^4 + 12\alpha_1^3\beta_1 + 8\alpha_1^2\beta_1^2 + 4\alpha_1^2\alpha_2 + 6\alpha_1\alpha_2\beta_1 + 4\alpha_2\beta_1^2 \\
&\qquad + 20\alpha_1^2\beta_2 + 24\alpha_1\beta_1\beta_2 + \alpha_1\alpha_2\gamma + \alpha_2\beta_1\gamma + \alpha_2^2 - 8\alpha_2\beta_2 + 16\beta_2^2 \\
&\qquad + (2 \alpha_1^2 + a2)(\gamma^2 + \beta_1 \gamma) + (3 \alpha_1^3 + 3 \alpha_1^2 \beta_1 + \alpha_1 \alpha_2 + \alpha_2 \beta_1 + 6 \alpha_1 \beta_2)(2\gamma) \\
&= 4\alpha_1^4 + 12\alpha_1^3\beta_1 + 8\alpha_1^2\beta_1^2 + 4\alpha_1^2\alpha_2 + 6\alpha_1\alpha_2\beta_1 + 4\alpha_2\beta_1^2 \\
&\qquad + 20\alpha_1^2\beta_2 + 24\alpha_1\beta_1\beta_2 + \alpha_1\alpha_2\gamma + \alpha_2\beta_1\gamma + \alpha_2^2 - 8\alpha_2\beta_2 + 16\beta_2^2.
\end{align*}

Next we excise the locus in $V_1(-1) \boxtimes W_{-2}$
where one of the linear forms is zero. In other words, we excise
from $(V_1(-1) \boxtimes W_{-2} \smallsetminus (0 \times 0)) / (G \times \gl_2)$ the closed substack $Z$ given by
\[\big(V_1(-1) \boxtimes (L_{-2} \times 0) \cup V_1(-1) \boxtimes (0 \times L_{-2}) \big) / (G \times \gl_2) \simeq (V_1(-1) \boxtimes (L_{-2} \times 0)) / (\mathbb{G}_m \times \mathbb{G}_m \times \gl_2).\]
As a module over $\chow^*(BG \times B\! \gl_2)$, the Chow ring $\chow^*(Z)$ is generated
by $1$ and $t_1$, so it suffices to determine the pushforwards of $1$ and $t_1$
to $\chow^*([V_1(-1) \boxtimes W_{-2} \smallsetminus (0 \times 0)] / (G \times \gl_2))$.
If we write $a_1$ and $a_2$ for the Chern roots of $V$,
then applying Lemma~\ref{lm:sub}, the class of
\begin{multline*}
[V_1(-1) \boxtimes (L_{-2} \times 0) \smallsetminus (0 \times 0)] / (\mathbb{G}_m \times \mathbb{G}_m \times \gl_2) \\
\subset [V_1(-1) \boxtimes (L_{-2} \oplus L_{-2}) \smallsetminus (0 \times 0)] / (\mathbb{G}_m \times \mathbb{G}_m \times \gl_2)
\end{multline*}
is given by
\begin{align*}
c_2(V_1(-1) \boxtimes (0 \times L_{-2})) &= (-\alpha_1 - 2t_2 - a_1)(-\alpha_1 - 2t_2 - a_2) \\
&= (\alpha_1 + 2t_2)^2 + (\alpha_1 + 2t_2)(a_1 + a_2) + a_1 a_2 \\
&= (\alpha_1 + 2t_2)^2 + (\alpha_1 + 2t_2)\alpha_1 + \alpha_2 \\
&= 4 t_2^2 + 6 \alpha_1 t_2 + 2\alpha_1^2 + \alpha_2.
\end{align*}
Excising this locus therefore introduces relations given by the pushforwards of
\[4 t_2^2 + 6 \alpha_1 t_2 + 2\alpha_1^2 + \alpha_2 \tand (4 t_2^2 + 6 \alpha_1 t_2 + 2\alpha_1^2 + \alpha_2)t_1\]
along $B(\mathbb{G}_m \times \mathbb{G}_m) \to BG$;
by Lemma~\ref{pushbg}, these are
\begin{align*}
4 \cdot (\beta_1^2 + \beta_1 \gamma - 2 \beta_2) + 6 \alpha_1 \cdot (\beta_1 + \gamma) + (2\alpha_1^2 + \alpha_2) \cdot 2 &= 4 \alpha_1^2 + 6 \alpha_1 \beta_1 + 4 \beta_1^2 + 2 \alpha_2 - 8 \beta_2 \\
&\qquad + (3 \alpha_1 + 2 \beta_1) \cdot 2 \gamma \\
&= 4 \alpha_1^2 + 6 \alpha_1 \beta_1 + 4 \beta_1^2 + 2 \alpha_2 - 8 \beta_2,
\end{align*}
respectively
\begin{align*}
4 \cdot (\beta_1 \beta_2 + \gamma \beta_2) + 6 \alpha_1 \cdot 2\beta_2 + (2\alpha_1^2 + \alpha_2) \cdot (\beta_1 + \gamma) &= 2 \alpha_1^2 \beta_1 + \alpha_2 \beta_1 + 12 \alpha_1 \beta_2 + 4 \beta_1 \beta_2 + \alpha_2 \gamma \\
&\qquad + (\alpha_1^2 + 2 \beta_2) \cdot 2 \gamma \\
&= 2 \alpha_1^2 \beta_1 + \alpha_2 \beta_1 + 12 \alpha_1 \beta_2 + 4 \beta_1 \beta_2 + \alpha_2 \gamma.
\end{align*}
To summarize, the relations obtained by excising the loci where one form is zero
are thus:
\begin{align}
\label{relzero}
\begin{split}
9 \alpha_2^2 - 2 \alpha_1^2 \alpha_2 &= 0 \\
4 \alpha_1^2 + 6 \alpha_1 \beta_1 + 4 \beta_1^2 + 2 \alpha_2 - 8 \beta_2 &= 0 \\
2 \alpha_1^2 \beta_1 + \alpha_2 \beta_1 + 12 \alpha_1 \beta_2 + 4 \beta_1 \beta_2 + \alpha_2 \gamma &= 0 \\
4\alpha_1^4 + 12\alpha_1^3\beta_1 + 8\alpha_1^2\beta_1^2 + 4\alpha_1^2\alpha_2 + 6\alpha_1\alpha_2\beta_1 + 4\alpha_2\beta_1^2 + 20\alpha_1^2\beta_2 \qquad & \\
\phantom{a}  + 24\alpha_1\beta_1\beta_2 + \alpha_1\alpha_2\gamma + \alpha_2\beta_1\gamma + \alpha_2^2 - 8\alpha_2\beta_2 + 16\beta_2^2 &= 0
\end{split}
\end{align}

\subsection{The tautological classes\label{subsec:taut}}
Since
\[H^0(\omega_C) = \left\langle \frac{x dy - y dx}{z}, \frac{x dy - y dx}{w} \right\rangle,\]
with the natural
action of $\gl_2$ on $(x, y)$ and $G$ on $(z, w)$,
the Hodge bundle is just the dual of the standard representation of $G$
tensored with the dual of the determinant representation of $\gl_2$.
In particular its Chern classes are
\begin{align*}
\lambda_1 &= - \beta_1 - 2\alpha_1 \\
\lambda_2 &= \alpha_1^2 + \alpha_1 \beta_1 + \beta_2.
\end{align*}

To find the pullback of $\delta_1$ to our family,
we note that the preimage of the boundary stratum $\Delta_1$
is the image of the
Segre map
$V_1(-1) \times W_{-2} \to V_1(-1) \boxtimes W_{-2}$.
So by Lemma~\ref{lm:segre},
\[\delta_1 = c_1(V_1(-1)) + c_1(W_{-2}) = -3 \alpha_1 - 2 \beta_1 + \gamma.\]

\subsection{Excision of triple root loci\label{subsec:triple}}
To excise the triple root loci, we will find the classes of the triple root loci
in the product of projectivizations
\[[\pp V_3(1) \times \pp(V_1(-1) \boxtimes W_{-2})] / (G \times \gl_2) \simeq [\pp V_3 \times \pp (V_1 \boxtimes W_{-2})] / (G \times \gl_2),\]
and pull back to $[V_3(1) \times (V_1(-1) \boxtimes W_{-2})] / (G \times \gl_2)$
by substituting the hyperplane class on $\pp V_3$ for $\alpha_1$,
and the hyperplane class on $\pp (V_1 \boxtimes W_{-2})$ for $-\alpha_1$.

We begin by excising the locus where the cubic form has a triple root.
Substituting the hyperplane class on $\pp V_3$ for $\alpha_1$,
the recursion relation from Lemma~\ref{rec} yields
\begin{align} \label{recv3}
\begin{split}
s^0_3 &= 1 \\
s^1_3 &= \alpha_1 \\
s^2_3 &= 3 \alpha_2 \\
s^3_3 &= \alpha_1 \alpha_2.
\end{split}
\end{align}
Consequently, we can apply
Lemma~\ref{ver} to calculate that excising this locus imposes relations given by the vanishing of
\[3 s_3^2 - 6\alpha_1 s_3^1 + 6(\alpha_1^2 - \alpha_2) = 3\alpha_2 \tand s_3^3 - 6 \alpha_2 s_3^1 + 6 \alpha_1 \alpha_2 = \alpha_1 \alpha_2.\]

We next excise the locus where all three forms share a common factor;
this is the image of the closed immersion
\[[\pp V_2 \times \pp V_1 \times \pp W_{-2}] / (G \times \gl_2) \to [\pp V_3 \times \pp (V_1 \boxtimes W_{-2})] / (G \times \gl_2)\]
defined by
\[(f, g, h) \mapsto (f \cdot g, g \boxtimes h).\]

Note that the hyperplane class on $\pp V_3$, respectively
on $\pp (V_1 \boxtimes W_{-2})$, pulls back to the
sum of the hyperplane classes on $\pp V_2$ and $\pp V_1$,
respectively on $\pp V_1$ and $\pp W_{-2}$.
Since these, along with the hyperplane class on $\pp W_{-2}$,
generate $\chow^*([\pp V_2 \times \pp V_1 \times \pp W_{-2}] / (G \times \gl_2))$
as a module over $\chow^*(BG \times B\!\gl_2)$,
the relations obtained by excision are all generated by
the pushforwards of the fundamental class $1$
and hyperplane class $w = c_1(\oo_{\pp W_{-2}}(1))$ under this map.

To compute these two pushforwards, we factor this map as
the composition of the diagonal map on the $\pp V_1$ factor
\begin{equation} \label{diagpv}
[\pp V_2 \times \pp V_1 \times \pp W_{-2}] / (G \times \gl_2) \to [\pp V_2 \times (\pp V_1 \times \pp V_1) \times \pp W_{-2}] / (G \times \gl_2)
\end{equation}
followed by multiplication and the Segre map
\begin{equation} \label{multsegre}
[(\pp V_2 \times \pp V_1) \times (\pp V_1 \times \pp W_{-2})] / (G \times \gl_2) \to [\pp V_3 \times \pp (V_1 \boxtimes W_{-2})] / (G \times \gl_2).
\end{equation}

Writing $x_1$ and $x_2$ for the hyperplane sections on the two $\pp V_1$ factors,
Lemma~\ref{diag} gives that
the pushforwards of $1$ and $w$ under \eqref{diagpv} are
$x_1 - \alpha_1 + x_2$ and $x_1 w - \alpha_1 w + x_2 w$.
Using Lemmas~\ref{lm:mult} and~\ref{lm:segre},
and writing $x = c_1(\oo_{\pp (V_1 \boxtimes W_{-2})}(1))$, the pushforwards
of these classes under \eqref{multsegre} are
\begin{multline*}
s_3^1 \cdot [2x + c_1(V_1) + c_1(W_{-2})] - \alpha_1 \cdot 3s_3^0 \cdot [2x + c_1(V_1) + c_1(W_{-2})] \\
+ 3s_3^0 \cdot [x^2 + c_1(W_{-2}) x + c_2(W_{-2}) - c_2(V_1)],
\end{multline*}
respectively
\begin{multline*}
s_3^1 \cdot [x^2 + c_1(V_1) x + c_2(V_1) - c_2(W_{-2})] - \alpha_1 \cdot 3s_3^0 \cdot [x^2 + c_1(V_1) x + c_2(V_1) - c_2(W_{-2})] \\
+ 3s_3^0 \cdot \big[x^3 + (c_1(V_1) + c_1(W_{-2})) x^2 + (c_2(V_1) + c_1(V_1) c_1(W_{-2}) + c_2(W_{-2})) x \\
+ c_1(V_1) c_2(W_{-2}) + c_2(V_1) c_1(W_{-2})\big].
\end{multline*}
From Lemma~\ref{rec}, we have that $s_3^0 = 1$, and that
$s_3^1$ is the hyperplane class on $\pp V_3$, which we must
substitute for $\alpha_1$. Substituting $x$ for
$-\alpha_1$, the relations we obtain
from excising this locus are the vanishing of
\begin{multline*}
\alpha_1 \cdot [-2\alpha_1 + c_1(V_1) + c_1(W_{-2})] - 3\alpha_1 \cdot [-2\alpha_1 + c_1(V_1) + c_1(W_{-2})] \\
+ 3[\alpha_1^2 - c_1(W_{-2}) \alpha_1 + c_2(W_{-2}) - c_2(V_1)],
\end{multline*}
respectively the vanishing of
\begin{multline*}
\alpha_1 \cdot [\alpha_1^2 - c_1(V_1) \alpha_1 + c_2(V_1) - c_2(W_{-2})] - 3\alpha_1 \cdot [\alpha_1^2 - c_1(V_1) \alpha_1 + c_2(V_1) - c_2(W_{-2})] \\
+ 3\big[-\alpha_1^3 + (c_1(V_1) + c_1(W_{-2})) \alpha_1^2 - (c_2(V_1) + c_1(V_1) c_1(W_{-2}) + c_2(W_{-2})) \alpha_1 \\
+ c_1(V_1) c_2(W_{-2}) + c_2(V_1) c_1(W_{-2})\big].
\end{multline*}

Since $V_1 = V^*$, we have $c_1(V_1) = -\alpha_1$ and $c_2(V_1) = \alpha_1$;
similarly since $W_{-2} = W_2^*$, Lemma~\ref{npower} gives $c_1(W_{-2}) = \gamma - 2\beta_1$
and $c_2(W_{-2}) = 4\beta_2$. 
Substituting these in and collecting like terms, our relations are simply the vanishing of
\[9 \alpha_1^2 + 10 \alpha_1 \beta_1 + \alpha_1 \gamma - 3 \alpha_2 + 12 \beta_2 - 3 \alpha_1 \cdot 2 \gamma = 9 \alpha_1^2 + 10 \alpha_1 \beta_1 + \alpha_1 \gamma - 3 \alpha_2 + 12 \beta_2,\]
respectively the vanishing of
\begin{multline*}
\alpha_2 \gamma - 10 \alpha_1^3 - 12 \alpha_1^2 \beta_1 - 5 \alpha_1 \alpha_2 - 6 \alpha_2 \beta_1 - 16 \alpha_1 \beta_2 + (3\alpha_1^2 + \alpha_2) \cdot 2 \gamma \\
= \alpha_2 \gamma - 10 \alpha_1^3 - 12 \alpha_1^2 \beta_1 - 5 \alpha_1 \alpha_2 - 6 \alpha_2 \beta_1 - 16 \alpha_1 \beta_2.
\end{multline*}

Finally, we excise the locus where the cubic form is divisible by the square of one
of the linear forms.
This locus is the norm from $\mathbb{G}_m \times \mathbb{G}_m$ to $G$
of the locus where the square of the first linear form divides the cubic form (this latter locus
is only $(\mathbb{G}_m \times \mathbb{G}_m \times \gl_2)$-equivariant).

Away from the locus where one of the linear forms is zero, we have a
$(\mathbb{G}_m \times \mathbb{G}_m \times \gl_2)$-equivariant natural map
$\pp (V_1 \boxtimes W_{-2}) \to \pp V_1 \times \pp V_1$ induced by projection
from the subspaces $V_1 \boxtimes (L_{-2} \times 0)$ and $V_1 \boxtimes (0 \times L_{-2})$.
This induces a map
\begin{equation} \label{changeofhyp}
\pp V_3 \times \pp (V_1 \boxtimes W_{-2}) \to \pp V_3 \times \pp V_1 \times \pp V_1;
\end{equation}
the locus we want to excise can then be described as the pullback under this map
of the image of the map
\begin{equation} \label{21}
\pp V_1 \times \pp V_1 \times \pp V_1 \to \pp V_3 \times \pp V_1 \times \pp V_1
\end{equation}
defined by
\[(f, g, h) \mapsto (f g^2, g, h).\]

Since the pullback under \eqref{changeofhyp}
of the hyperplane class on the first $\pp V_1$ factor
differs from the hyperplane class on $\pp (V_1 \boxtimes W_{-2})$
by $2t_2$, we just need to push forward the generators
of $\chow^*([\pp V_1 \times \pp V_1 \times \pp V_1] / (\mathbb{G}_m \times \mathbb{G}_m \times \gl_2))$
along \eqref{21}, substitute the hyperplane class on $\pp V_3$ for $\alpha_1$,
substitute the hyperplane class on the first $\pp V_1$ factor for $-\alpha_1 - 2t_2$,
and take norms along $B(\mathbb{G}_m \times \mathbb{G}_m) \to BG$.
Moreover, since pullback along \eqref{21} is surjective by inspection,
the only generators we need to use are $1$ and $t_1$
(which generate $\chow^*(B(\mathbb{G}_m \times \mathbb{G}_m))$ as a module over $\chow^*(BG)$).

To implement this, we factor \eqref{21} as a composition
of a triple diagonal map on the middle $\pp V_1$ factor
\begin{equation} \label{35}
\pp V_1 \times \pp V_1 \times \pp V_1 \to \pp V_1 \times (\pp V_1 \times \pp V_1 \times \pp V_1) \times \pp V_1
\end{equation}
followed by multiplication
\begin{equation} \label{fmul}
(\pp V_1 \times \pp V_1 \times \pp V_1) \times \pp V_1 \times \pp V_1 \to \pp V_3 \times \pp V_1 \times \pp V_1.
\end{equation}

Write $x_i$ (with $1 \leq i \leq 5$) for the hyperplane class on the $i$th $\pp V$ factor.
From Lemma~\ref{diag3}, the pushforward of the fundamental class
\[x_2 x_3 + x_3 x_4 + x_4 x_2 - (x_2 + x_3 + x_4) \alpha_1 + \alpha_1^2 - \alpha_2.\]
Using Lemma~\ref{lm:mult} and \eqref{recv3}, and writing $x$ for the hyperplane class on $\pp V_1$
that we must substitute for $-\alpha_1 - 2t_2$,
the pushforward of the fundamental class along \eqref{fmul} is then
\begin{align*}
1 &\mapsto s_3^2 + 2x \cdot 2s_3^1 - 2\alpha_1 \cdot 2 s_3^1 - \alpha_1 x \cdot 6 s_3^0 + (\alpha_1^2 - \alpha_2) \cdot 6 s_3^0 \\
&= 3\alpha_2 + 2(-\alpha_1 - 2t_2) \cdot 2 \alpha_1 - 2\alpha_1 \cdot 2 \alpha_1 - \alpha_1 (-\alpha_1 - 2t_2) \cdot 6 + (\alpha_1^2 - \alpha_2) \cdot 6 \\
&= 4 \alpha_1 t_2 + 4 \alpha_1^2 - 3\alpha_2.
\end{align*}
Our relations are thus given by the vanishing of the
norms along $B(\mathbb{G}_m \times \mathbb{G}_m) \to BG$ of the classes
\[4 \alpha_1 t_2 + 4 \alpha_1^2 - 3\alpha_2 \tand (4 \alpha_1 t_2 + 4 \alpha_1^2 - 3\alpha_2) t_1.\]
Applying Lemma~\ref{pushbg}, these are
\[4 \alpha_1 (\beta_1 + \gamma) + 8 \alpha_1^2 - 6\alpha_2 = 4 \alpha_1 \beta_1 + 8 \alpha_1^2 - 6\alpha_2 + 2 \cdot 2\gamma = 4 \alpha_1 \beta_1 + 8 \alpha_1^2 - 6\alpha_2,\]
respectively
\begin{align*}
4 \alpha_1 \cdot 2 \beta_2 + (4 \alpha_1^2 - 3\alpha_2)(\beta_1 + \gamma) &= 8 \alpha_1 \beta_2 + 4 \alpha_1^2 \beta_1 - 3\alpha_2 \beta_1 + \alpha_2 \gamma + (2\alpha_1^2 - 2 \alpha_2) \cdot 2 \gamma \\
&= 8 \alpha_1 \beta_2 + 4 \alpha_1^2 \beta_1 - 3\alpha_2 \beta_1 + \alpha_2 \gamma.
\end{align*}
To summarize, the relations obtained by excising the triple root loci are thus:
\begin{align}
\label{reltrip}
\begin{split}
3\alpha_2 &= 0 \\
\alpha_1 \alpha_2 &= 0 \\
4 \alpha_1 \beta_1 + 8 \alpha_1^2 - 6\alpha_2 &= 0 \\
8 \alpha_1 \beta_2 + 4 \alpha_1^2 \beta_1 - 3\alpha_2 \beta_1 + \alpha_2 \gamma &= 0 \\
9 \alpha_1^2 + 10 \alpha_1 \beta_1 + \alpha_1 \gamma - 3 \alpha_2 + 12 \beta_2 &= 0 \\
\alpha_2 \gamma - 10 \alpha_1^3 - 12 \alpha_1^2 \beta_1 - 5 \alpha_1 \alpha_2 - 6 \alpha_2 \beta_1 - 16 \alpha_1 \beta_2 &= 0.
\end{split}
\end{align}

\subsection{Chow ring of our test family}
From our expressions for the tautological classes on $\mm$ in Section~\ref{subsec:taut},
we can express $\alpha_1$, $\beta_1$, and $\beta_2$ in terms of $\gamma$,
$\delta_1$, $\lambda_1$, and $\lambda_2$:
\begin{align*}
\alpha_1 &= -2 \lambda_1 + \delta_1 - \gamma \\
\beta_1 &= 3 \lambda_1 - 2 \delta_1 + 2\gamma \\
\beta_2 &= \lambda_2 - \alpha_1^2 - \alpha_1 \beta_1 \\
&= 2 \lambda_1^2 - 3 \lambda_1 \delta_1 + \delta_1^2 + 3 \lambda_1 \gamma - 2 \delta_1 \gamma + \gamma^2 + \lambda_2.
\end{align*}

Subtracting the relation
$4 \alpha_1^2 + 6 \alpha_1 \beta_1 + 4 \beta_1^2 + 2 \alpha_2 - 8 \beta_2 = 0$
(see \eqref{relzero})
from the relation
$3 \alpha_2 = 0$ (see \eqref{reltrip}),
we obtain
\[\alpha_2 = 4 \alpha_1^2 + 6 \alpha_1 \beta_1 + 4 \beta_1^2 - 8 \beta_2 = 2 \lambda_1 \delta_1 - 2 \lambda_1 \gamma - 8 \lambda_2.\]

Substituting these expressions into all of our previous relations
in \eqref{relzero}, \eqref{reltrip}, and Theorem~\ref{thm:bg},
we check the ideal the resulting relations generate is as follows:

\begin{thm} The Chow ring of the moduli space $B$ of bielliptic
curves is generated by $\gamma$, $\delta_1$, $\lambda_1$, and $\lambda_2$,
subject to the relations:
\begin{align*}
2 \gamma &= 0 \\
\gamma^2 + \lambda_1 \gamma &= 0 \\
\delta_1^2 + \delta_1 \gamma + 8 \lambda_1^2 - 12 \lambda_2 &= 0 \\
24 \lambda_1^2 - 48 \lambda_2 &= 0 \\
2 \delta_1^2 + 2 \lambda_1 \delta_1 &= 0 \\
20 \lambda_1 \lambda_2 - 4 \delta_1 \lambda_2 &= 0 \\
8 \lambda_1^3 - 8 \lambda_1 \lambda_2 &= 0.
\end{align*}
\end{thm}

\noindent
In particular, reducing modulo $2$, we obtain
\[\chow^*(B) \otimes \zz/2\zz \simeq \zz[\gamma, \delta_1, \lambda_1, \lambda_2] / (\gamma^2 + \lambda_1 \gamma, \delta_1^2 + \delta_1 \gamma),\]
from which we conclude that the classes
\[\delta_1(\delta_1 + \lambda_1) \lambda_1^2, \quad \delta_1(\delta_1 + \lambda_1) \lambda_2, \tand \delta_1(\delta_1 + \lambda_1) (\lambda_1^2 + \lambda_2)\]
are all nonvanishing in $\chow^*(B) \otimes \zz/2\zz$, and thus in $\chow^*(\mm) \otimes \zz/2\zz$ as desired.

\bibliographystyle{alpha}
\bibliography{m2bar}

\begin{thebibliography}{MVW06}

\bibitem[Blo86]{bloch}
Spencer Bloch.
\newblock Algebraic cycles and higher {$K$}-theory.
\newblock {\em Adv. in Math.}, 61(3):267--304, 1986.

\bibitem[EG98]{EdGr}
Dan Edidin and William Graham.
\newblock Equivariant intersection theory.
\newblock {\em Invent. Math.}, 131(3):595--634, 1998.

\bibitem[Fab90]{faber}
Carel Faber.
\newblock Chow rings of moduli spaces of curves. {I}. {T}he {C}how ring of
  {$\overline{M}_3$}.
\newblock {\em Ann. of Math. (2)}, 132(2):331--419, 1990.

\bibitem[Mum83]{enum}
David Mumford.
\newblock Towards an enumerative geometry of the moduli space of curves.
\newblock In {\em Arithmetic and geometry, {V}ol. {II}}, volume~36 of {\em
  Progr. Math.}, pages 271--328. Birkh\"auser Boston, Boston, MA, 1983.

\bibitem[MVW06]{pt}
Carlo Mazza, Vladimir Voevodsky, and Charles Weibel.
\newblock {\em Lecture notes on motivic cohomology}, volume~2 of {\em Clay
  Mathematics Monographs}.
\newblock American Mathematical Society, Providence, RI; Clay Mathematics
  Institute, Cambridge, MA, 2006.

\bibitem[Pap07]{pappas}
Georgios Pappas.
\newblock Integral {G}rothendieck-{R}iemann-{R}och theorem.
\newblock {\em Invent. Math.}, 170(3):455--481, 2007.

\bibitem[Tot14]{totaro}
Burt Totaro.
\newblock {\em Group cohomology and algebraic cycles}, volume 204 of {\em
  Cambridge Tracts in Mathematics}.
\newblock Cambridge University Press, Cambridge, 2014.

\bibitem[Vis98]{vistoliM2}
Angelo Vistoli.
\newblock The {C}how ring of {$M_2$}. {A}ppendix to ``{E}quivariant
  intersection theory'' [{I}nvent.\ {M}ath.\ {\bf 131} (1998), no.\ 3,
  595--634; {MR}1614555 (99j:14003a)] by {D}. {E}didin and {W}. {G}raham.
\newblock {\em Invent. Math.}, 131(3):635--644, 1998.

\end{thebibliography}

\end{document}